\documentclass[11pt, oneside]{article}

\usepackage{amsmath, amssymb, amsthm}
\usepackage{geometry}
\usepackage{epsfig}
\usepackage[dvipsnames]{xcolor}
\usepackage[font=small,labelfont=bf, width=.85\textwidth]{caption}
\usepackage[shortlabels]{enumitem}
\usepackage[normalem]{ulem} 

\usepackage{tikz}
\usetikzlibrary{shapes.geometric}
\usetikzlibrary{decorations.markings}

\newtheorem{theorem}{Theorem}

\newtheorem{proposition}[theorem]{Proposition}
\newtheorem{lemma}[theorem]{Lemma}

\newtheorem{corollary}[theorem]{Corollary}
\newtheorem{conjecture}[theorem]{Conjecture}
\newtheorem{observation}[theorem]{Observation}
\newtheorem{claim}{Claim}

\newenvironment{proofc}{\vspace*{-.22in}\begin{proof}[Proof of Claim]}{\end{proof}}

\newcommand{\Z}{\mathbb{Z}}

\begin{document}

\title{Group connectivity in 3-edge-connected signed graphs}
\author{
Alejandra Brewer Castano\footnote{azb0163@auburn.edu; Department of Mathematics and Statistics, Auburn University, Auburn, AL, USA 36849}
\qquad
Jessica McDonald\footnote{mcdonald@auburn.edu; Department of Mathematics and Statistics, Auburn University, Auburn, AL, USA 36849. Supported in part by Simons Foundation grant \#845698.}
\qquad
Kathryn Nurse\footnote{corresponding author; knurse@sfu.ca; School of Computing Science, Simon Fraser University, Burnaby, B.C., Canada V5A 1S6. Supported in part by the Natural Sciences and Engineering Research Council of Canada (NSERC).}}
\date{}

\maketitle

\begin{abstract}Jaeger, Linial, Payan, and Tarsi introduced the notion of $A$-connectivity for graphs in 1992, and proved a decomposition for cubic graphs from which $A$-connectivity follows for all 3-edge-connected graphs when $|A|\geq 6$. The concept of $A$-connectivity was generalized to signed graphs by Li, Luo, Ma, and Zhang in 2018 and they proved that all 4-edge-connected flow-admissible signed graphs are $A$-connected when $|A|\geq 4$ and $|A|\neq 5$. 
We prove that all 3-edge-connected flow-admissible signed graphs are $A$-connected when $|A|\geq 6$ and $|A|\neq 7$. Our proof is based on a decomposition that is a signed-graph analogue of the decomposition found by Jaeger et. al, and which may be of independent interest.
\end{abstract}


\section{Introduction}

In this paper, graphs are finite and may contain multiple edges and loops. We are concerned with group connectivity in signed graphs; the unfamiliar reader may refer to the relevant definitions in Section 2 of this paper. For notation and terminology not defined here we follow \cite{WestText}. 

Seymour's 6-Flow Theorem \cite{Seymour6flow} says that every 2-edge-connected graph has a nowhere-zero $6$-flow. Bouchet \cite{Bouchet} conjectured an analog for signed graphs, namely that every flow-admissible signed graph has a nowhere-zero 6-flow. DeVos, Li, Lu, Luo, Zhang and Zhang \cite{DLLZZ} showed that such signed graphs have a nowhere-zero 11-flow; Bouchet's Conjecture has also been verified for $6$-edge-connected graphs by Xu and Zhang \cite{XZ}, and for $4$-edge-connected graphs by Raspaud and Zhu \cite{RASPAUD2011464}.

Due to seminal work of Tutte \cite{tutte1949,tutte_1954, tutte_1956}, we know that a graph has a nowhere-zero $k$-flow iff it has a nowhere-zero $A$-flow for any abelian group $A$ with $|A|=k$, and so Seymour's 6-flow Theorem implies that every 2-edge-connected graph has a nowhere-zero $A$-Flow, for every abelian group $A$ with $|A|\geq 6$. This equivalence is not true in either direction for signed graphs (see Li, Luo, Ma, and Zhang \cite{LLMZ}), and so the following conjecture would extend, but not imply, Bouchet's Conjecture.

\begin{conjecture}\label{conj: Bouchet group analog} Every flow-admissible signed graph has a nowhere-zero $A$-flow for every abelian group $A$ with $|A|\geq 6$.
\end{conjecture}

In this paper we contribute a result in support of Conjecture \ref{conj: Bouchet group analog}. To get our result we assume more than just flow-admissibility, and we don't get the result when $|A|=7$ --- but for all other $|A|\geq 6$, we actually prove that the signed graph is \emph{$A$-connected} rather than just having a nowhere-zero $A$-flow. Namely, we prove the following.

\begin{theorem}\label{main} Let $G$ be a $3$-edge-connected, 2-unbalanced signed graph. Then $G$ is $A$-connected for every abelian group $A$ with $|A|\geq 6$ and $|A|\neq 7$.
\end{theorem}

The notion of $A$-connectivity is a powerful generalization of nowhere-zero $A$-flows: in particular, given a (signed) graph $G$, if a subgraph $H$ of $G$ is $A$-connected, then $G/H$ has a nowhere-zero $A$-flow iff $G$ does (Jaeger, Lineal, Payan, Tarsi \cite{JaegerFrancois1992Gcog}; Li, Luo, Ma, Zhang \cite{LLMZ}). Lov\'asz, Thomassen, Wu and Zhang \cite{LTWZ} were able to take advantage of the inductive strength $A$-connectivity provides in order to show that 6-edge-connected graphs have nowhere-zero $\mathbb{Z}_3$-flows (and hence satisfy Tutte's 3-Flow Conjecture); in fact they proved that 6-edge-connected graphs are $\mathbb{Z}_3$-connected in order to achieve this goal.

Jaeger, Linial, Payan, and Tarsi \cite{JaegerFrancois1992Gcog} were the first to generalize the concept of nowhere-zero $A$-flows to $A$-connectedness. Noting that 2-edge-connectivity does not imply $A$-connectivity for any $A$,
they proved that every 3-edge-connected graph is $A$-connected for every abelian group $A$ with $|A|\geq 6$, thus providing a group-connectivity analog to Seymour's 6-flow Theorem. 
To prove their theorem, Jaeger et. al found a pretty decomposition of a $3$-connected cubic graph (minus a vertex) into a $1$-base and a $2$-base\footnote{as defined in \cite{Seymour6flow}, and here in Section \ref{sect:decompositions}.} that are disjoint. Similarly, our Corollary \ref{cor: decomp general JLPT} gives a decomposition of a certain cubic signed graph into a $1$-base and a $2$-base (minus an induced negative cycle) that are disjoint.

In the same paper, Jaeger et. al proved that every 4-edge-connected graph is $A$-connected for every abelian group $A$ with $|A|\geq 4$. Li, Luo, Ma, and Zhang \cite{LLMZ} extended the notion of $A$-connectedness to signed graphs, and proved that every 4-edge-connected, 2-unbalanced signed graph is $A$-connected for every abelian group $A$ with $|A|\geq 4$ and $|A|\neq 5$. 
Note that 2-unbalanced is required for flow-admissibility in signed graphs (see Section 2 for more on this). Our Theorem \ref{main} takes the next step from this last result by lowering the edge-connectivity as far as possible for such a statement about group-connectivity, thereby getting closer to the assumptions of Conjecture \ref{conj: Bouchet group analog}. 

Even in graphs, group-connectivity is not in general a monotone property (for example Jaeger, Linial, Payan, and Tarsi \cite{JaegerFrancois1992Gcog} provided a graph that is $\mathbb{Z}_6$-connected but not $\mathbb{Z}_5$-connected). Given this, it is worth stating that the $|A|\geq 6$ bound in Theorem \ref{main} is best possible. In Section 4 of this paper we  discuss a 2-unbalanced signature of the Petersen graph that is not $A$-connected for any abelian group $A$ with $|A|\leq 5$, and indeed does not even have a nowhere-zero $A$-flow for any such $A$. Interestingly, this example was first presented by Bouchet \cite{Bouchet} who proved that although it is flow admissible, it does not have a nowhere zero $k$-flow for any $k\leq 5$. This means that there is a common sharpness example for our Theorem \ref{main}, for Conjecture \ref{conj: Bouchet group analog}, and for Bouchet's Conjecture.

In addition to our sharpness example, Section 4 discusses the Flow--Colouring Duality Theorem of Tutte \cite{tutte_1954}, and how this can be extended to other surfaces via signed graphs. We include here statements and proofs that may be considered implicit from the work of Tutte \cite{tutte_1954} and Bouchet \cite{Bouchet} but seem otherwise absent in the literature. We use such results to justify our sharpness example, but also to prove an important case of Theorem \ref{main} in this section.

Our proof of Theorem \ref{main} is inspired by Seymour's second proof of his 6-Flow Theorem (sketched in Section 5 of \cite{Seymour6flow}). Seymour's proof contains a structure theorem about cubic 3-edge-connected graphs, and DeVos \cite{Matt} recently noticed that a similar theorem exists for signed graphs. Section 5 contains two structure theorems: DeVos's theorem, as well as a second structure theorem which is a signed graph analogue of the structure found by Jaeger et. al \cite{JaegerFrancois1992Gcog}. In order to apply these fruitfully, we need to make some reductions for our problem, in particular to cubic graphs. These reductions are the subject of Section 3. It is worth noting that while Seymour's reduction from 2-edge-connectivity to 3-edge-connectivity is impossible for signed graphs (and for group connectivity), 
we are able to reduce our problem using the idea of so-called ``contractible configurations'', and in particular a result of Li, Luo, Ma, and Zhang \cite{LLMZ} that says, under some conditions on signed graphs $G, H$  with $H\subseteq G$, that $G$ is $A$-connected iff $G/H$ is $A$-connected.

The sixth and final section of this paper has the proof of Theorem \ref{main}, with two subsections giving the arguments for the remaining two cases (one case having been handled previously in Section 4). Note that given a signed graph, and two abelian groups $A, B$ with $|A|=|B|$, a nowhere-zero $A$-flow does not necessarily imply the existence of a nowhere-zero $B$-flow (see Li, Luo, Ma, and Zhang \cite{LLMZ}). However, in our proof of Theorem \ref{main}, it is only the size and not the structure of the groups that is important for our cases. On the other hand, our cases do divide based on the structure of the signed graph, and in particular, whether our not it has two disjoint negative cycles. We'll actually prove that Theorem \ref{main} extends to the case $|A| = 7$ unless $G$ has two disjoint negative cycles. While we suggest it should be possible to extend Theorem 2 to this last case, it appears well out of reach of our methods.

\section{Signed graphs and group connectivity}

A \emph{signed graph} $(G,\sigma)$ is a graph $G$ together with a \emph{signature} $\sigma: E(G) \to \{-1,1\}$. The signature partitions $E(G)$ into \emph{positive} edges, namely those $e$ with $\sigma(e)=1$, and \emph{negative} edges, namely those $e$ with $\sigma(e)=-1$. For convenience we usually refer to a signed graph as $G$, omitting the name of the signature if no confusion arises. Signed graphs with only positive edges can also be interpreted as unsigned graphs, and all the definitions below are consistent with this interpretation. Moreover, in the absences of specific definition for signed graphs, we use the same definition for graphs (eg. a signed graph $G$ is \emph{simple} if it has no loops or parallel edges). 

Given a signed graph $(G, \sigma)$ and a vertex $v\in V(G)$, we perform a \emph{switch} operation at $v$ to get a new signature $\sigma'$ by defining $\sigma'(e) = -\sigma(e)$ for all edges $e$ incident to $v$. We use the term \emph{switch on an edge-cut} in $G$ to mean performing the switch operation on all vertices on one side of an edge-cut; this has the effect of changing the signature of every edge in the edge-cut and leaving the signature of every other edge unchanged.
We say that two signatures of $G$ are \emph{equivalent} if one can be obtained from the other by a sequence of switch operations, or equivalently, by a single switch on an edge-cut. A signed graph is \emph{balanced} if there is an equivalent signature where all edges are positive, and \emph{unbalanced} otherwise. A signed graph $G$ is \emph{$k$-unbalanced} if every equivalent signature of $G$ has at least $k$ negative edges. A cycle $C$ in a signed graph is called \emph{negative} (or \emph{unbalanced}) if $\prod_{e\in H} \sigma(e)=-1$; otherwise $C$ is called \emph{positive} (or \emph{balanced}). It is well known that two signatures are equivalent iff they have the same set of negative cycles \cite{ZaslavskyThomas1982Sg}. In particular, this means that any balanced subgraph $H$ of a signed graph can always be re-signed so that every edge of $H$ is positive.

Every edge of a signed graph can be thought of as being composed of two half-edges $h$ and $h'$, each of which is incident with one end \cite{Bouchet}. Given a signed graph $G$ and a positive edge $e$ with endpoints $x,y$, \emph{orienting} $e$ means choosing exactly one of the half-edges to direct away from its endpoint, with the other half-edge directed towards its endpoint. On the other hand, if the edge $e$ is negative, then \emph{orienting} $e$ means choosing to either direct both half-edges away from their endpoints, or direct both half-edges towards their endpoints. Equivalently, we can define $\tau(h) =1$ if $h$ is a half-edge oriented away from its endpoint and $\tau(h) =-1$ if $h$ is a half-edge oriented towards its endpoint. In this way we may refer to an \emph{orientation} of a signed graph (where every edge has been oriented) as $\tau$; such an oriented signed graph is also referred to as a \emph{bidirected graph}.

Given a signed graph $G$ with orientation $\tau$, an abelian group $A$, and a function $f: E(G) \to A$, the \emph{boundary} of $f$ at a vertex $v$ is 
$$\partial f (v)=\sum_{h\in H_G(v)} \tau (h) f(e_h),$$
where $H_G(v)$ is the set of half-edges incident to $v$, and $e_h$ is the edge containing $h$.
Note that $\partial f(v)$ amounts to summing the $f$-values of all edges pointing out of $v$, and subtracting the $f$-values of all edges pointing into $v$.  
The function $f$ is an \emph{$A$-flow} of $G$ if $\partial f (v)=0$ for each $v\in V(G)$, and is an (integer) \emph{$k$-flow} if it is a $\mathbb{Z}$-flow and $|f(e)|< k$ for each $e\in E(G)$.
A flow $f$ is \emph{nowhere zero}, abbreviated \emph{nz}, if $f(e)\neq 0$ for all $e\in E(G)$. 

Observe that for a signed graph $G$, abelian group $A$, and mapping $f: E(G) \to A$, 
$$\sum_{v\in V(G)}\partial f (v)=2a$$
for some $a\in A$. This is because a positive edge contributes 0 to the sum (it adds to the boundary of one endpoint and subtracts the same value from their other endpoint), and a negative edge $e$ either contributes $2f(e)$ or $-2f(e)$ to this sum (according to whether its half-edges point away from or towards both their endpoints, respectively). If $G$ is a signed graph with all edges positive (or unsigned graph), this means that $\sum_{v \in V(G)}\partial f(v) =0$. In general we call any mapping $\beta : V(G) \rightarrow A$ satisfying $\sum_{v \in V(G)} \beta(v) =2a$ for some $a\in A$ an \emph{A-boundary} of the signed graph $G$. 

Given a signed graph $G$ and an $A$-boundary $\beta$ we say that we can \emph{satisfy} $\beta$ if there exists a function $f : E(G) \rightarrow A\setminus\{0\}$ and orientation $\tau$ so that $\partial f = \beta$. A signed graph $G$ is \emph{A-connected} for some abelian group $A$ if we can satisfy all $A$-boundaries of $G$. Since $\beta=0$ is always an $A$-boundary of $G$, an $A$-connected graph has, in particular, an nz $A$-flow. 

Given a signed graph $G$, an abelian group $A$, and an $A$-boundary $\beta$, we can satisfy $\beta$ iff we can satisfy $\beta$ using any orientation of $G$; this is because the flow (and signature) is maintained when simultaneously changing the orientation of an edge $e$ and replacing $f(e)$ with its inverse element. Similarly, if we \emph{switch} at a vertex $v$ by changing the direction (but not the $f$- value) of every half-edge incident to $v$, this changes the signature of every edge incident to $v$ and so $\partial f(v)$ is now the negative of what it was. If we can satisfy this new slightly adjusted boundary, then we can still satisfy the original $\beta$ by switching at $v$ again. Thus, we can satisfy an $A$-boundary of a signed graph $G$ iff we can satisfy that (adjusted) $A$-boundary for any equivalent signature of $G$. We will repeatedly take advantage of this fact in this paper by freely choosing any orientation and equivalent signature of a given signed graph. We will also use the following result. 

\begin{proposition}[Li, Luo, Ma, Zhang \cite{LLMZ}]\label{equivAcon} Let $G$ be a connected 2-unbalanced signed graph and let $A$ be an abelian group. Then the
following statements are equivalent:
\begin{enumerate}
\item $G$ is $A$-connected.
\item Given any $\overline{f} : E(G) \rightarrow A$ and orientation of $G$, there exists an $A$-flow $f : E(G) \rightarrow A$ with this same orientation such that $f(e)\neq \overline{f}(e)$ for every $e\in E(G)$.
\end{enumerate}
\end{proposition}

Note that while there is a fixed orientation in Proposition \ref{equivAcon}(2), following the above discussion we are free to choose any equivalent signature, as well as any orientation of our signed graph $G$, provided we adjust $\overline{f}$ accordingly.

The assumption of 2-unbalanced in the above proposition is standard when discussing $A$-connectivity in signed graphs. This is because it can easily be argued that a signed graph with exactly one negative edge cannot have an nz $A$-flow for any odd-ordered $A$. 
A signed graph is \emph{flow-admissible} if it admits a nz $k$-flow for some positive integer $k$. Bouchet \cite{Bouchet} proved that a connected signed graph is flow-admissible iff it is 2-unbalanced and there is no cut edge $e$ such that $G-e$ has a balanced component.  
In fact, a signed graph with such a cut edge cannot have an nz $A$-flow for any abelian group $A$.
Hence, flow-admissibility is a natural assumption to make in Conjecture \ref{conj: Bouchet group analog}.

Before proceeding to the main content of our paper, it may be useful to think about how one could assign flow in a signed graph.  Given a negative cycle $C$ with one negative edge, if we try to assign a flow $f=a$ in a consistent direction around $C$, we will get one vertex $v$ with  $\partial f(v) =\pm 2a$. This is always non-zero when the group has odd order. However, if we pair up two such negative cycles, joined by a path, then we can get a flow on the overall graph by putting a flow of value $2a$ on the path directed from one cycle to the other. Indeed, given our rules about switching, this works for any \emph{barbell}, that is, any union of two unbalanced cycles joined by a (possibly trivial) path. It is not difficult to prove that any nz $A$-flow (or nz integer flow, see eg. Bouchet \cite{Bouchet}) on a signed graph $G$ is the sum of nz $A$-flows (or nz integer flow) on positive cycles or barbells when $|A|$ is odd; when $|A|$ is even, we must add negative cycles to this list.
Our preference when assigning flows is to find a positive cycle, and one way of doing this is to find a \emph{theta} subgraph, namely a set of three internally-disjoint $(x,y)$-paths for some pair of vertices $x,y$ (which can be drawn to look like the Greek letter theta). Regardless of the signature on a theta, at least one of the ${3 \choose 2}$ cycles created by the three paths will be positive. Also important in terms of assigning flow is the notion of a \emph{base} of a signed graph, which is a maximal spanning subgraph which contains neither balanced cycles nor barbells. A base in a signed graph is an analog of a spanning tree in an unsigned graph, where the addition of a single edge creates a structure that may be used to assign flow. In particular, if $G$ is a connected signed graph, then a \emph{connected base} of $G$ is a spanning tree if $G$ is balanced, or a spanning tree plus one extra edge forming a negative cycle if $G$ is unbalanced.

\section{Reductions}

This section establishes certain properties of a minimum counterexample to Theorem \ref{main}; the aim of this section is the following lemma. For the lemma, and the remainder of the document, if $G$ is a signed graph and $X\subseteq V(G)$, we denote by $\delta(X)$ the set of edges with exactly one endpoint in $X$; when $X=\{v\}$ we write $\delta(v)$ in place of $\delta(\{v\})$.

\begin{lemma}\label{lem:reduction}
Let $G$ be a $3$-edge-connected, 2-unbalanced signed graph, and let $A$ be an abelian group with $|A|\geq 6$. Suppose $G$ is not $A$-connected and suppose that $G$ is chosen to be minimum with respect to $\sum_{v \in V(G)}|\deg(v)-3|$, and subject to that with respect to $|V(G)|$. Then $G$ is simple, cubic, and 3-connected. Moreover, if there exists $X \subseteq V(G)$ with $G[X]$ balanced and either 
\begin{enumerate}[(i)]
    \item $|\delta(X)| =3$, or
    \item $|\delta(X)| \leq 5$, and $G[X]$ is planar,
\end{enumerate}  
then $G[X]$ is a (possibly trivial) path.
\end{lemma}

We prove Lemma~\ref{lem:reduction} at the end of this section, after establishing certain preliminaries. 

For a signed graph G and edge $e\in E(G)$ the \emph{contraction $G/e$} is obtained from $G$ by identifying the two ends of $e$, and for any loops created by this identification, deleting them if they are positive and keeping them if they are negative. Given $X\subseteq E(G)$ (or $X\subseteq G$), the signed graph $G/X$ is obtained from $G$ by contracting every edge in $X$ (or $E(X)$).

Our first result towards Lemma \ref{lem:reduction} is about a reversal of the contraction operation on positive edges. Suppose $G$ is a signed graph, $v \in V(G)$ with $\deg(v) \geq 4$, and $e,f$ are distinct edges incident to $v$ whose other ends are $v_e, v_f$ respectively. We \emph{uncontract at $v$ with $\{e,f\}$} to get a new signed graph by adding new vertex $v'$, changing the ends of $e,f$ to be $v_ev'$ and $v_fv'$ respectively, and adding a positive edge $vv'$. See Figure \ref{fig:uncontract}. 

\begin{figure}[htb]
    \centering
    \includegraphics[height=2.2cm]{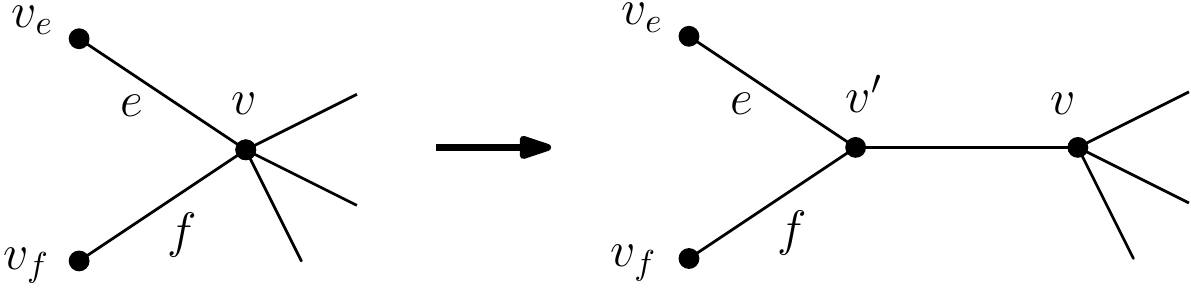}
    \caption{Uncontracting at $v$ with $\{e,f\}$.}
    \label{fig:uncontract}
\end{figure}

\begin{proposition}\label{prop:can uncontract and keep 2-unbal 3ec}
Let $G$ be a 2-unbalanced, 3-edge-connected signed graph on at least two vertices. Let $v \in V(G)$ have degree at least four, and $e \in E(G)$ be incident to $v$. Then there is an edge $e'$ incident to $v$ so that the graph $G'$, obtained from $G$ by uncontracting at $v$ with $\{e,e'\}$, is 2-unbalanced and 3-edge-connected.
\end{proposition}

\begin{proof}
This proof is in two claims.

\setcounter{claim}{0}

\begin{claim}\label{claim:at most 1 edge not 2-unbal}
There is at most one edge $f\neq e$ incident to $v$ so that the signed graph obtained from $G$ by uncontracting at $v$ with $\{e,f\}$ is not 2-unbalanced.
\end{claim}
 
\begin{proofc}
Suppose for a contradiction that two such edges exist, say $f,f'$. The new graphs have a signature with exactly one negative edge, which must be the uncontracted edge. (Otherwise that signature restricted to $G$ is a signature of $G$ with only one negative edge.) This means there is a signature $\sigma$ of $G$ whose set of negative edges is $\{e,f\}$, and a signature  $\sigma'$ of $G$ whose set of negative edges is $\{e, f'\}$. But, the edges on which $\sigma$ and $\sigma'$ differ must be an edge-cut, by the definition of switching. Hence $\{f,f'\}$ is an edge-cut of size 2 in $G$, a contradiction.
\end{proofc}

\begin{claim}\label{claim:two edges keep 3ec}
There exist two edges $f_1, f_2 \neq e$ incident to $v$ such that the graphs obtained from $G$ by uncontracting at $v$ with $\{e,f_1\}$ and $\{e, f_2\}$ are 3-edge-connected.
\end{claim}

\begin{proofc}
    If the graph after uncontracting is not 3-edge-connected, then any edge-cut of size at most 2 must use the uncontracted edge. Consider $G-v$. If $G-v$ is disconnected then let $S$ be the component of $G-v$ containing the other end of $e$. By 3-edge-connectivity of $G$, there are at least three edges from $v$ to each component of $G-v$. Choosing edges $f_1,f_2$ incident to $v$ with an end in some component $S' \neq S$ satisfies the claim. If $G-v$ has a cut-edge $d$, let $S$ be the component of $G-v-d$ containing the other end of $e$. Similar to before, choosing edges $f_1,f_2$ incident to $v$ with an end in some component $S' \neq S$ satisfies the claim. Finally, if $G-v$ is $2$-edge-connected, then any two edges $f_1,f_2 \neq e$ incident to $v$ work.
\end{proofc}

These two claims complete the proof of Proposition~\ref{prop:can uncontract and keep 2-unbal 3ec}.
\end{proof}

\begin{observation}\label{obs:G' Aconn}
In Proposition \ref{prop:can uncontract and keep 2-unbal 3ec}, if $G'$ is $A$-connected, then so is $G$.
\end{observation}
\begin{proof}
Let $\beta$ be any $A$-boundary on $G$, and extend to an $A$-boundary $\beta'$ on $G'$ by setting $\beta(v')=0$. Since $G'$ is $A$-connected there is a function $f'$ satisfying $\partial f' =\beta'$. But since the uncontracted edge $vv'$ is positive, $f'|_{E(G)}$ satisfies $\beta$ on $G$.
\end{proof}

\begin{lemma}\label{lem: reduce to cubic}
Let $A$ be an abelian group where $|A| \geq 2$. Let $G$ be a 2-unbalanced, 3-edge-connected signed graph. Suppose that $G$ is not $A$-connected and, subject to that, $G$ is minimal with respect to $\sum_{v \in V(G)}|\deg(v)-3|$. Then $G$ is simple, cubic, and 3-connected.
\end{lemma}

\begin{proof}
If $|V(G)| =1$, then all edges in $E(G)$ are loops, two of which must be negative. Since any $A$-boundary $\beta$ satisfies $\beta(v)=2a$ for some $a\in A$, we can always satisfy the boundary by making use of those two negative loops. Hence $G$ must be $A$-connected, a contradiction. So we may assume $|V(G)| \geq 2$.

Since $G$ is 3-edge-connected, no vertex has degree less than 3. Suppose for contradiction $v \in V(G)$ has $\deg(v) \geq 4$, and let $e$ be an edge incident to $v$.
By Proposition \ref{prop:can uncontract and keep 2-unbal 3ec}, there exists an edge $f$ incident to $v$ so that the graph $G'$
obtained from $G$ by uncontracting at $v$ with $\{e,f\}$ is 2-unbalanced and 3-edge-connected. Note that the new vertex $v'$ incident to $e, f$ in $G'$ now contributes zero to the sum we had minimized, and $v$ contributes one less, hence $G'$ is $A$-connected by minimality of $G$. But now $G$ is $A$-connected by Observation \ref{obs:G' Aconn}, a contradiction. Therefore $G$ is cubic.

Since $G$ is cubic, $|V(G)|$ is even. If $|V(G)| = 2$, then $G$ has just three edges which are parallel, so it cannot be $2$-unbalanced. Therefore $|V(G)| \ge 4$. Since $G$ is cubic and 3-edge-connected, it has no loops or parallel edges, and is therefore simple. Finally, it is straightforward to see that a cubic, 3-edge-connected graph on at least $4$ vertices must also be 3-connected.
\end{proof}

In the remainder of this section we employ the following two results, both of which were discussed but not stated precisely in the introduction.

\begin{theorem}[Li, Luo, Ma, Zhang \cite{LLMZ}
]\label{thm:flow-contract configs} Let $G, H$ be signed graphs with $H\subseteq G$ and such that each of $G, H$ is either balanced or 2-unbalanced. Let $A$ be an abelian group, and suppose that $H$ is $A$-connected. Then $G$ is $A$-connected if and only if $G/H$ is $A$-connected.
\end{theorem}

\begin{theorem}[Jaeger, Linial, Payan, Tarsi \cite{JaegerFrancois1992Gcog}
]\label{thm: JLPT G-v is A-con}
If $G$ is a 3-edge-connected, balanced signed graph, $A$ is an abelian group of order at least 6, and $v \in V(G)$ is a vertex of degree 3, then $G-v$ is $A$-connected.
\end{theorem}

In the following lemma, by a \emph{minimal $2$-edge-boundary} in a signed graph $G$, we mean a set $X \subseteq V(G)$ such that $|\delta(X)| = 2$, and so that for all $Y \subseteq X$, $|\delta(Y)| \neq 2$.

\begin{lemma}\label{lem:balanced at most 5 vert of degree 2 is A con}
Let $A$ be an abelian group with $|A| \ge 6$, and let $G$ be a planar, $2$-edge-connected, balanced signed graph so that the number of minimal 2-edge-boundaries in $G$ is at most 5. 
Then $G$ is $A$-connected.
\end{lemma}
\begin{proof}
Suppose for contradiction that the lemma is false. Let $G$ be an edge-minimum counterexample to the lemma.
Clearly $G$ has at least one edge, otherwise it is just an isolated vertex and is trivially $A$-connected.

Choose a facial cycle $C \subseteq G$ with minimum number of edges.
By the last assumption of the lemma, $G$ has at most $5$ vertices of degree~$2$, the rest have degree at least $3$. By these degree conditions and Euler's formula,
it follows that $C$ has at most $5$ edges. And since $G$ is balanced, $C$ is positive. So, given any $\overline{f}: E(C) \rightarrow A$, there exists an $A$-flow $f$ such that $f(e) \neq \overline{f}(e)$ for every $e \in E(C)$ because $|A| > |E(C)|$. Hence $C$ is $A$-connected by Proposition~\ref{equivAcon}. 
Theorem~\ref{thm:flow-contract configs} then tells us that since $G$ is not $A$-connected, neither is $G/C$. 
But since $C$ is a facial cycle, $G/C$ is a planar graph. Note that any edge-cut in $G/C$ is an edge-cut in $G$, and so $G/C$ is $2$-edge-connected. Also, any minimal $2$-edge-boundary in $G/C$ either is a minimal $2$-edge-boundary in $G$, or contains the vertex into which $C$ was contracted and that contraction eliminated a minimal $2$-edge-boundary from the graph. Because minimal $2$-edge-boundaries must be disjoint, $G/C$ has at most as many minimal $2$-edge-boundaries as $G$. Hence $G/C$ satisfies all the assumptions of the lemma, and contradicts the minimality of $G$.
\end{proof}

With Lemmas~\ref{lem: reduce to cubic} and \ref{lem:balanced at most 5 vert of degree 2 is A con} in place, and employing Theorems~\ref{thm:flow-contract configs} and \ref{thm: JLPT G-v is A-con}, we prove the main result of this section.

\begin{proof}[Proof of Lemma \ref{lem:reduction}]
Let $G$ be a signed graph which satisfies the assumptions of the lemma.
By Lemma \ref{lem: reduce to cubic}, $G$ is simple, cubic, and $3$-connected. We are left to show that the last sentence of the lemma holds, and we proceed by contradiction.  To this end, suppose there exists $X \subseteq V(G)$ with $G[X]$ balanced, which satisfies either (i) or (ii), and $G[X]$ is not a path. Suppose that $X$ is vertex-minimal subject to these conditions.

If $X$ satisfies (i), then $G[X]$ can be obtained from a 3-edge-connected balanced signed graph by deleting one vertex of degree $3$. (For example, the underlying graph formed by identifying $V(G)\setminus X$ to a single vertex $v$ is 3-edge-connected. Make this graph a signed graph where every edge is positive, and delete $v$.)  This means $G[X]$ is $A$-connected by Theorem~\ref{thm: JLPT G-v is A-con}. Theorem \ref{thm:flow-contract configs} then tells us that since since $G$ is not $A$-connected, neither is $G/X$. But $G/X$ is 3-edge-connected, 2-unbalanced, and cubic so $\sum_{v \in V(G/X)}|\deg(v)-3| = 0$ (as it was for $G$). So since $v(G/X) < v(G)$, $G/X$ contradicts the choice of $G$.

It must be that $X$ satisfies (ii) but not (i), thus $|\delta(X)| \in \{4,5\}$. Then $G[X]$ is a balanced, planar signed graph which is not a path. We will show that it is additionally 2-edge-connected. We can see that $G[X]$ is connected by the 3-connectivity of $G$, and since $|\delta(X)|\leq 5$. Suppose however that there is an edge-cut $\delta(Y)$ in $G[X]$ with $|\delta_{G[X]}(Y)| \leq 1$. One side of this cut, say $G[Y]$, is not a (possibly trivial) path. But then since $G[Y]$ is planar with $|\delta_G(Y)| \leq 5$, this contradicts the minimality of $X$. Hence $G[X]$ is $2$-edge-connected.

Since $G$ is cubic, and $|\delta(X)| \in \{4,5\}$, it follows that $G[X]$ has exactly four or exactly five vertices of degree~$2$ (one for each edge in $\delta(X)$), the rest have degree~$3$. By $3$-edge-connectivity of $G$, every edge-cut of size at most two in $G[X]$ contains a vertex of degree-$2$. This means number of minimal 2-edge-cuts in $G$ is at most 5.

We may now apply Lemma \ref{lem:balanced at most 5 vert of degree 2 is A con}, to get that $G[X]$ is $A$-connected. Because $G$ is not $A$-connected, this means, by Theorem \ref{thm:flow-contract configs}, that $G/X$ is also not $A$-connected. Now, $G/X$ is 2-unbalanced and 3-edge-connected, but no longer cubic. Indeed, the vertex $v$ that $X$ was contracted into has $\deg_{G/X}(v) \in \{4,5\}$, but all other vertices have degree~3. By Proposition \ref{prop:can uncontract and keep 2-unbal 3ec}, we can form a cubic, 2-unbalanced, 3-edge-connected (and hence 3-connected) signed graph $G'$ from $G/X$ by uncontracting at $v$ (possibly twice). Observation \ref{obs:G' Aconn} tells us that $G'$ is not $A$-connected. But $G[X]$ has at least four vertices, and at most two uncontractions happened, so $v(G') < v(G)$. Thus $G'$ contradicts the minimality of~$G$.
\end{proof}
    
\section{Projective Planar Duality}\label{sec: ppduality}

In the plane, Tutte's Theorem on Flow--Colouring Duality \cite{tutte_1954} says that a planar graph $G$ has a nz $k$-flow iff its planar dual $G^*$ has a $k$-colouring (for any integer $k\geq 2$); given Tutte's earlier-cited work we can also replace ``nz $k$-flow'' in this statement with ``nz $A$-flow'' for any abelian group $A$ with $|A|= k$. The proof of this duality theorem involves a bijection between orientations of an embedded graph $G$ and its planar dual $G^*$ achieved via the so-called ``right-hand rule''. This relationship extends to other surfaces via signed graphs, as we will detail now.

Consider a 2-cell embedding of an oriented graph $G$ in some surface $\Sigma$. Since every face is homeomorphic to a disc, every face may be equipped with an orientation in one of two ways (i.e. clockwise or counter-clockwise). Following Bouchet \cite{Bouchet}, and Mohar and Thomassen \cite{MoharThomassenText}, the \emph{dual} of this oriented $G$ on $\Sigma$ is an oriented signed graph $G^*$ embedded on $\Sigma$ with a vertex corresponding to each face of $G$, and with vertices in $G^*$ joined by an edge $e^*$ iff the two faces share an edge $e$. To determine the orientation (and signature) of the edge $e^*$ of $G^*$, we look at whether the orientation of the corresponding primal edge $e$ agrees or disagrees with its two incident faces. If $e$ agrees with exactly one of its incident faces $f$, then make $e^*$ a positive edge directed towards $f$. Otherwise, make $e^*$ a negative edge directed either towards both its endpoints (if $e$ agrees with both incident faces), or away from both its endpoints (if $e$ agrees with neither incident face). We say that such $G, G^*$ are \emph{oriented duals on $\Sigma$}.

Note that in the above definition, different notions of clockwise for the faces of $G$ lead to different signatures for $G^*$. In the plane it is convention to take a consistent clockwise orientation on all faces, which generates a unique (all-positive) dual, but in general we should say that a signed graph $G^*$ is \emph{a} dual of the embedded graph $G$, rather that \emph{the} dual. However, changing the orientation of a given face $f$ of $G$ amounts precisely to switching the signature of $G^*$ at the vertex $v_f$ corresponding to the face $f$, so all duals of $G$ have equivalent signatures.

Another point to make about the above definition is that, while we can take the dual of every graph $G$ embedded in a surface, it is not true that every signed graph embedded on a surface $S$ can be obtained as a dual of a graph embedded in $S$. Indeed Zaslavsky \cite{Zas} has proved a forbidden-minor characterization for  \emph{projective planar signed graphs}, that is, signed graphs which, up to switching, can be obtained as the dual of a graph embedded in the projective plane. 

The above duality concept is one of the main motivations cited by Bouchet \cite{Bouchet} for studying flows in signed graphs. The following lemma, which in particular generalizes the afore-mentioned Flow--Colouring Duality Theorem of Tutte, may be considered implicit from the work of Tutte \cite{tutte_1954} and Bouchet \cite{Bouchet}. However since explicit statements and proofs are difficult to find in the literature, we include it here.

\begin{lemma}\label{duality} Let $G, G^*$ be oriented duals on a surface $\Sigma$, with $e^* \in E(G^*)$ denoting the dual edge to each $e \in E(G)$,  and let $A$ be an abelian group with $|A| \geq 2$. Then:
\begin{enumerate}
\item If $c: V(G) \rightarrow A$, then  $f: E(G^*)\rightarrow A$ given by $f(e^*)=c(v)-c(u)$ where $e$ is directed from $u$ to $v$ in $G$, is an $A$-flow.
\item If $f: E(G^*)\rightarrow A$ is an $A$-flow, and if either $\Sigma$ is the plane or $\Sigma$ is the projective plane and $A$ has no element of order 2, then there exists $c: V(G) \rightarrow A$ such that $c(v)-c(u) = f(e^*)$ for every directed edge $e=uv \in E(G)$.
\end{enumerate}
\end{lemma}

\begin{proof} (1) We must show that $f$ satisfies the conservation of flow property at every vertex $x\in V(G)$. The edges incident to $x$ in $G^*$ correspond to a cycle $C$ in $G$. There is a notion of clockwise for $C$ given by the orientation of the face in $G$ corresponding to $x$, which allows us to think of each oriented edge in $C$ as either forwards or backwards. According to the way we have defined oriented duals, a forwards edge $e$ in $C$ corresponds to an edge $e^*$ directed towards $x$, and a backwards edge $e$ in $C$ corresponds to an edge $e^*$ directed outwards from $x$. (Note that such an edge $e^*$ may be negative, but we are only concerned with the orientation on its half incident to $x$.)
Let $T(C)$ be the total obtained by adding $f$ on the forwards edges of $C$, and subtracting $f$ on the backwards edges of $C$. Then $T(C)$ is also equal to the sum obtained by adding $f$ on all the edges pointing into $v$, and subtracting $f$ on all the edges pointing out of $v$.
But $T(C)$ amounts to adding $c(u)-c(v)$ for every forward edge $e=uv$ in $C$, and subtracting this value for every backwards edge. This means each vertex $v$ on $C$ contributes both $c(v)$ and $-c(v)$, so $T(C) = 0$ and $f$ is indeed a flow.

(2) Given the flow $f$ on $G^*$, we define $f^*: E(G)\rightarrow A$ by $f^*(e)=f(e^*)$. Let $W$ be any closed walk in $G$. Choose a direction for $W$, and let $f^*(W)$ be the sum computed by adding $f^*(e)$ for every time $e$ appears as a forwards edge in $W$, and subtracting $f^*(e)$ for every time $e$ appears as a backwards edge in $W$. We must show that $f^*(W)=0$; this will mean that by starting with any vertex $w\in G$ and setting $c(w)=0$, we can proceed to assign a $c$-value to each vertex in $G$ (repeating this process if $G$ has multiple components) so that $c(v)-c(u) =f^*(e)=f(e^*)$ for every directed edge $e=uv \in E(G)$.

If $W$ is non-contractible in the projective plane, then the walk we get by following twice along $W$, denoted by $2W$, is a contractible curve.  If we know that $f^*(2W)=0$ then this implies that $f^*(W)=0$, since in this case $A$ has no element of order 2 by assumption. So it suffices to assume that $W$ is contractible.

Since $W$ is a closed walk that is contractible in the plane or projective plane, it can be transformed into a trivial walk $W'$ (with $f^*(W')=0$) by two operations: removing immediate reversals (deleting both the forwards and backwards occurrence of the immediately reversed edge), and; replacing a segment of edges in $W$ who form a segment of a facial cycle $C$ in $G$ with the rest of that facial cycle, taken backwards. In the former case, the removal has the effect of subtracting $f^*(e)-f^*(e)=0$ from $f^*(W)$ for some edge $e$. In the latter case, to show that the replacement has no effect on the sum $f^*(W)$, we must show that $f^*(C)=0$. To this end, note that the set $E(C)$ corresponds exactly to the set of edges incident to some vertex $v\in V(G^*)$. According to the way we have defined oriented duals, the sum $f^*(C)$ is exactly equal to the net flow of $f$ at $v$, which is zero since $f$ is a flow. 
\end{proof}

Lemma \ref{duality} gives Tutte's flow-colouring duality theorem for $\Sigma$ the plane when the flows $f$ in question are nowhere-zero and the functions $c$ in question are proper vertex-colourings. Lemma \ref{duality} also shows that this duality of nz flows and proper vertex-colourings extends generally to the projective plane; when considering other surfaces we only get the direction of (1). Rather than considering nz flows in Lemma \ref{duality}, we can as well assume that the flows $f$ in question avoid a set of particular values $\overline{f}: E(G^*) \rightarrow A$ on the oriented $G^*$ (i.e. for group-connectivity as in Proposition \ref{equivAcon}). This in turn corresponds to a function $c: V(G)\rightarrow A$ where $c(v)-c(u) \neq \overline{f}(e^*)$ for every directed edge $e=uv \in E(G)$. This $c$ is a group-valued version of DP-colouring (see eg. \cite{BKP}, \cite{DP}).

Consider now the example given in Figure \ref{fig: PetersenK6alt}. Here we see the graph $K_6$ drawn, with circle vertices and thin lines, on the projective plane (with the dotted circle indicating the cross-cap). Note that a dual, drawn with sold vertices and bolded lines, will be a signed version of the Petersen graph, with edges though the cross-cap consisting of the edge-set of the 5-cycle $(v_1, v_2, v_3, v_4, v_5)$. We will give an orientation to each face of $K_6$ so that we may fix a specific signature and define the signed graph $\mathcal{P}_s$ as a dual of this embedding of $K_6$. To this end, give a clockwise orientation to each of the vertices in $\mathcal{P}_s$. The resulting facial orientation of $K_6$ makes an edge of $\mathcal{P}_s$ negative iff it goes through the cross-cap. (Note that this special relationship occurred here because faces involving the cross-cap are each bordered by non-cross-cap edges only on the side of the cross-cap where we have drawn the vertices $v_i$.) Of course, by our comments above, a different orientation assigned to the faces of $K_6$ can result in a different signature, but all such signatures will be equivalent to that of $\mathcal{P}_s$.

\begin{figure}[htb]
    \centering
    \includegraphics[height=6.5cm]{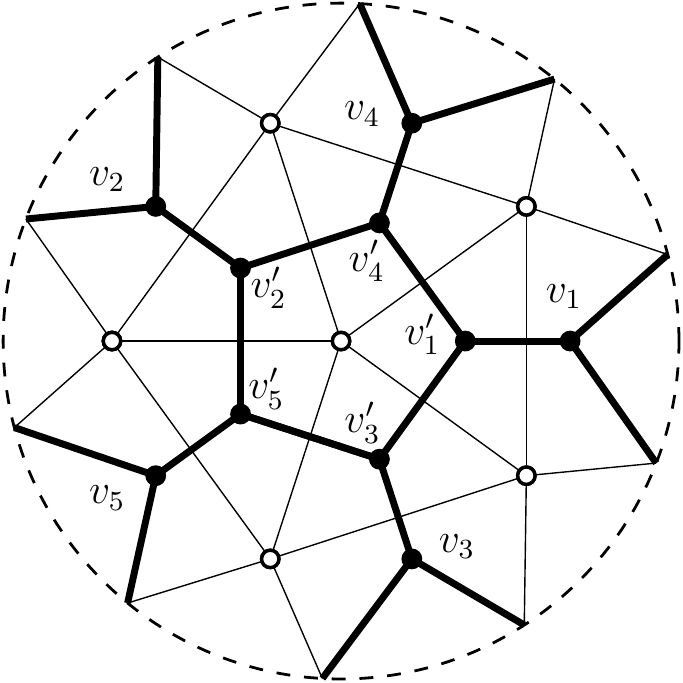}
    \caption{A drawing of $K_6$ on the projective plane (with circle vertices and normal lines), and its projective planar dual $\mathcal{P}_s$ (with solid vertices and bolded lines) where negative edges are precisely the edge-set of the cycle $(v_1, v_2, \ldots, v_5)$.}
    \label{fig: PetersenK6alt}
\end{figure}

Note that $\mathcal{P}_s$ is 3-edge-connected and 2-unbalanced, and in particular is flow-admissible. However, Bouchet \cite{Bouchet} proved that it has no nz $k$-flow for any $k\leq 5$, and he showed that this is because  $K_6$ is not $k$-colourable for any such $k$. Hence $\mathcal{P}_s$ shows sharpness for Bouchet's Conjecture. Using Lemma \ref{duality}(2) we now see that this idea extends to show that $\mathcal{P}_s$ is sharp for both Conjecture \ref{conj: Bouchet group analog} and our Theorem \ref{main}.

\begin{lemma} Let $A$ be an abelian group with $|A|\leq 5$. Then the signed graph $\mathcal{P}_s$ is not $A$-connected, and in particular has no nz $A$-flow. Moreover, $\mathcal{P}_s$ has no nz $k$-flow for any $k\leq 5$.
\end{lemma}

\begin{proof} First, suppose for a contradiction that $\mathcal{P}_s$ has a nz $k$-flow $f$ for some $k\leq 5$. By reversing orientations as necessary we may assume that $f(e)\in\{1, 2, 3, 4\}$ for all $e\in E(\mathcal{P}_s)$. We apply Lemma \ref{duality}(2) with $A=\mathbb{Z}$ (since $\mathcal{P}_s$ is projective planar), yielding $c: V(K_6) \rightarrow \mathbb{Z}$ such that $c(v)-c(u) = f(e^*)$ for every directed edge $e=uv \in E(K_6)$. Let $c'(w) \equiv c(w)$ (mod 5) for all $w\in V(K_6)$. Since $f$ is nowhere zero, this function $c'$ is a 5-colouring of $K_6$, a contradiction. Therefore $\mathcal{P}_s$ has no nz $k$-flow for any $k\leq 5$.

It remains to show the first part of the lemma holds. The result for $|A|\in\{3, 5\}$ follows from Lemma \ref{duality}(2) since $K_6$ is not 5-colourable and such an $A$ cannot have any element of order 2. It is obvious that $\mathcal{P}_s$ has no nz $\mathbb{Z}_2$-flow, since it has vertices of odd degree. In fact, since $\mathcal{P}_s$ is cubic, any nz $\mathbb{Z}_2\times \mathbb{Z}_2$-flow would need to have exactly one of each of the nonzero elements incident to each vertex, which would be a 3-edge-colouring of $\mathcal{P}$, a contradiction.

Finally, suppose for a contradiction that $\mathcal{P}_s$ has an nz $\mathbb{Z}_4$-flow, say $\varphi$. Then every vertex must be incident to exactly one edge with $\varphi$-value 2. Without loss of generality, there are exactly two different perfect matchings of $\mathcal{P}_s$ on which these 2's could appear; using the vertex-labelling from Figure \ref{fig: PetersenK6alt}, these are: $\{(v_i, v'_i): 1\leq i\leq 5\}$, or $\{(v_1, v'_1), (v_2', v_4'), (v_5', v_3'), (v_2, v_3), (v_4, v_5)\}$. In either case, deleting such a perfect matching leaves us with two vertex-disjoint cycles of length 5, one of which is positive. We may assume, without loss, that such a cycle has all positive edges and is oriented as a directed cycle, in which case its edges must alternate between $\varphi$-value 1 and $\varphi$-value 3. This is a contradiction, since the cycle has length 5.
\end{proof}

In order to show that a signed graph is not $A$-connected, as we just did with $\mathcal{P}_s$, we used part (2) of Lemma \ref{duality}. Let us now use part (1) of Lemma \ref{duality} to prove the following important case of our Theorem \ref{main}. 

\begin{theorem}\label{thm: no 2 disj negative cycles}
Let $G$ be a $3$-edge-connected, 2-unbalanced signed graph, and let $A$ be an abelian group with $|A|\geq 6$. Suppose $G$ is not $A$-connected and suppose that $G$ is chosen to be minimum with respect to $\sum_{v \in V(G)}|\deg(v)-3|$, and subject to that with respect to $|V(G)|$. Then $G$ must contain two disjoint negative cycles.
\end{theorem}

We will see how Theorem \ref{thm: no 2 disj negative cycles} fits within the rest of our Theorem \ref{main} later, in Section \ref{sect:proof of main}. For now, let us simply provide a proof, aided by the following result of Slilaty \cite{Slilaty}. In fact, Slilaty has completely characterized signed graphs without two disjoint negative cycles; we state below a version for 3-edge-connected cubic graphs only, which will be sufficient for us.

\begin{theorem}[Slilaty \cite{Slilaty}]\label{thm: Slilaty} Let $G$ be a 3-edge-connected, cubic signed graph, and suppose that $G$ has no two disjoint negative cycles. Then either:
\begin{enumerate}
\item there is an an equivalent signature with at most one negative edge; or
\item there is an equivalent signature where the set of negative edges in $G$ forms a triangle; or
\item $G$ is projective planar.
\end{enumerate}
\end{theorem}

\begin{proof}[Proof of Theorem \ref{thm: no 2 disj negative cycles}]
Suppose for a contradiction that $G$ does not contain two disjoint negative cycles, and apply Theorem \ref{thm: Slilaty}. Since $G$ is 2-unbalanced, situation (1) from Theorem \ref{thm: Slilaty} is impossible. By Lemma \ref{lem:reduction}, $G$ is cubic, so (2) would imply a $3$-edge-cut meeting the conditions of Lemma \ref{lem:reduction} (i), in particular that that the graph induced by the vertices outside of the triangle is a path. But then since $G$ is cubic, the path must be trivial and $G$ must be $K_4$ (with one triangle signed all negative). Then $G$ has a positive cycle $C \subseteq G$ of length $4$. It is straightforward to verify that $C$ is $A$-connected by Proposition \ref{equivAcon} (as we did in the proof of Lemma \ref{lem:balanced at most 5 vert of degree 2 is A con}) since $|A|\geq 5$. Then $G/C$ is a single vertex with two negative loops; by oppositely orienting the loops and assigning any common value to both, we see that $G/C$ is $A$-connected by Proposition \ref{equivAcon} since $|A|\geq 3$. But now $G$ is $A$-connected by Theorem \ref{thm:flow-contract configs}.

We may now assume that $G$ satisfies (3) from Theorem \ref{thm: Slilaty} . Let $H$ be an oriented graph embedded in the projective plane with oriented dual $H^*=G$. Let $\overline{f}: E(G) \rightarrow A$; we must show there exists a flow $f$ on $G$ with $f(e)\neq \overline{f}(e)$ for all $e\in E(G)$. According to Lemma \ref{duality}(1) and Proposition \ref{equivAcon}, to prove $A$-connectivity it suffices to show that there exists $c: V(H) \rightarrow A$ such that $c(v)-c(u) \neq \overline{f}(e^*)$ for every directed edge $e=uv \in E(H)$.

Since $H$ is embedded in the projective plane, $H$ has a vertex $v_n$ with degree at most $5$, by Euler's formula. Order the vertices of $H$ as $v_0, ...., v_n$ so that $v_i$ is a vertex of degree at most $5$ in $H-\{v_{i+1}, \ldots, v_n\}$ for $0 \leq i \leq n-1$. To define the function $c$, first assign an arbitrary value from $A$ to $v_0$, and continue to assign $c(v_1), c(v_2), \ldots..$ in order so as not to violate our desired condition. When we come to some $v_i$, note that there are at most~$5$ edges incident to $v_i$ whose other end already has a $c$-value assigned. So there may be up to 5 different values that we need to avoid assigning to $c(v_i)$, but since $|A|\geq 6$, some legal choice is always possible. 
\end{proof}

Note that Theorem \ref{thm: no 2 disj negative cycles} includes the case $|A|=7$, which we actually omit from the statement of Theorem \ref{main}. This is because we don't know how to handle $|A|=7$ when the signed graph in question \emph{does} have two disjoint negative cycles. We are able to handle all other cases with $|A|\geq 6$ however, and we start working towards this in the following section.

\section{Decompositions}\label{sect:decompositions}

The main inspiration for our work lies in a decomposition of Seymour concerning \emph{$k$-bases}. These are are a refinement of the \emph{bases} concept that we discussed at the end of Section~2. Bases will  play a big role in our work starting in the latter half of this section; for now we will focus on $k$-bases and decompositions.

Seymour showed that a 3-connected cubic graph has a 1-base and a 2-base that are disjoint \cite{Seymour6flow}.\footnote{This construction is in the second proof, in Section 5 of \cite{Seymour6flow}. Note that this is different from the construction in the first proof, in Sections 3 and 4. In the first proof, a 2-base $B$ is found which is the edge-set of a collection of disjoint cycles. But $\overline{B}$ is not necessarily a 1-base.} DeVos noticed that a similar partition exists for signed graphs. To this end, we extend the definitions of $k$-closure and $k$-base to signed graphs in a way that is convenient for our purposes. 

If $G$ is a signed graph and $S \subseteq E(G)$, then the \emph{$k$-closure} $\langle S \rangle_k$ of $S$ is the minimal subset of edges $S' \subseteq E(G)$ such that $S \subseteq S'$, and there does not exist a positive cycle $C$ such that $1 \leq |E(C)\setminus S'| \leq k$. A \emph{$k$-base} is a subset $B \subseteq E(G)$ such that $\langle B \rangle_k = E(G)$. Note that in a $2$-edge-connected signed graph, a minimal 1-base is a base. 

We will also require a result on peripheral cycles due to Tutte \cite{TutteW.T.1963HtDa}. A \emph{peripheral cycle} in a (signed) graph $G$ is a cycle $C \subseteq G$  which is induced and such that $G-C$ is connected. Note that if $G$ is simple and cubic, a cycle $C$ is peripheral iff $G-E(C)$ is connected. Tutte showed that in 3-connected graphs, peripheral cycles generate the cycle space; namely he proved the following result. Here we use $\triangle$ to denote symmetric difference.

\begin{theorem}[Tutte \cite{TutteW.T.1963HtDa} 2.5]\label{thm: periph cyc generate cycle space}
Let $G$ be a 3-connected (signed) graph, and $C \subseteq G$ be a cycle. Then there exist peripheral cycles $P_1, P_2, \dots, P_k$ so that $E(C) = \triangle_{i = 1}^k E(P_i)$.
\end{theorem}

\begin{corollary} \label{cor: neg peiph cyc}
If $G$ is an unbalanced, 3-connected signed graph, then $G$ has a negative peripheral cycle.
\end{corollary}
\begin{proof}
Let $C$ be a negative cycle of $G$. Let $P_1, P_2, \dots, P_k$ be a list of peripheral cycles so that $E(C) = \triangle_{i = 1}^k E(P_i)$ by Theorem \ref{thm: periph cyc generate cycle space}. Note that the symmetric difference of positive cycles contains an even number of negative edges. Since $C$ has an odd number of negative edges, there must be a negative peripheral cycle in the list.
\end{proof}

The following theorem is a signed version of Seymour's decomposition due to DeVos.
The proof is similar to the unsigned version, but, crucially, when constructing the decomposition one must begin with a negative peripheral cycle. 
The proof of Theorem \ref{thm: decomposition into tree and 2-base} will be necessary for Theorem \ref{thm: decomposition to connected base and almost 2-base}.
Theorem \ref{thm: decomposition into tree and 2-base} is also used in Section \ref{subsect:composite}, to prove our main result for groups of composite order.
For convenience, in the remainder of the document we will view a subset of edges of a signed graph as also a signed graph itself, whose vertex-set is the ends of those edges. We apologize for this abuse of notation.

\begin{theorem}[DeVos \cite{Matt}]\label{thm: decomposition into tree and 2-base}
If $G$ is a 3-connected, cubic signed graph, then there exists a partition $X_1\cup X_2$ of $E(G)$ such that $X_1$ is a spanning tree and $X_2$ is a 2-base.
\end{theorem}

\begin{proof}
We will follow the line Seymour sketched at the end of \cite{Seymour6flow}.
Choose a partition (with possibly empty parts) $A \cup B \cup C$ of $E(G)$ so that 
\begin{enumerate}[(a)]
    \item $A \cup B$ is a 2-connected graph;
    \item\label{prop: C} $C$ is connected, and all vertices of $C$ have degree 1 or 3 in $C$;
    \item\label{prop: tree} $A \cup C$ contains a spanning tree of $G$;
    \item\label{prop: 2-base} The 2-closure of $B$ contains $A$.
    \item\label{property: -cyc in B} $B$ contains a cycle, which is negative if $G$ is unbalanced.
\end{enumerate}

First, we show that such a partition exists. To this end, choose a peripheral cycle $D$ by Theorem \ref{thm: periph cyc generate cycle space}. If $G$ is unbalanced, choose $D$ to be negative by Corollary \ref{cor: neg peiph cyc}. Then $A = \emptyset, B = E(D),$ and $C = E(G) \setminus E(D)$ is a partition satisfying the requirements. 

Choose such a partition with $C$ minimal. If $C$ is empty, then the theorem is proved because we can choose a spanning tree $X_1 \subseteq A$ by property \ref{prop: tree}, and let $X_2 = E(G) \setminus X_1$, which is a 2-base by property \ref{prop: 2-base}.

Suppose for contradiction that $C \neq \emptyset$.

\setcounter{claim}{0}
\begin{claim}\label{Claim: path P in C}
There is a path $P \subseteq C$ whose ends have degree 1 in $C$, and so that $C - E(P)$ has at most three components, two of which are isolated vertices. 
\end{claim}

\vspace{-13pt}
\begin{proof}[Proof of claim]
Following Tutte, we will refer to the connected components of $C - E(P)$ that are not an end of $P$ as \emph{bridges}.
If $C = K_2$, the claim follows easily. We assume to the contrary. Hence $C$ has at least three vertices of degree~1 (and is connected, by \ref{prop: C}). We must find a path $P$ with exactly one bridge. 

Let the size of a bridge $\Gamma$ be $|V(\Gamma)| + |E(\Gamma)|$. Choose $P$ to lexicographically maximize the sizes of the bridges. That is, the biggest bridge has largest possible size, and subject to that, the second biggest bridge has largest possible size, etc. This ordering has the advantage that increasing the size of any bigger bridge at the expense only of smaller bridges improves the ordering. With this in mind, let $\Gamma$ be the smallest bridge and suppose, for a contradiction, that $\Gamma$ is not the only bridge. Since $C$ has vertices only of degree 1 or 3, every internal vertex $v$ of $P$ has exactly one edge going to some bridge. We will say that bridge \emph{attaches} at $v$. 

Suppose first that the attachments of $\Gamma$ are not consecutive along $P$, as in Figure \ref{fig:LexBridges} (a). Then there must be a path $P' \subseteq P \cup \Gamma$ whose ends are the ends of $P$, and which avoids some attachment vertex of some other bridge $\Gamma'$ (which is bigger than $\Gamma$). This path $P'$ contradicts our choice of $P$.

So it must be that the attachments of $\Gamma$ are consecutive along $P$, as in Figure \ref{fig:LexBridges} (b). This means the two edges of $P$ occurring immediately before and immediately after the attachments of $\Gamma$ form a 2-edge-cut in $C$. Because $G$ is 3-edge-connected, there must be a vertex $v$ of degree 1 in $\Gamma$. Also, since $\Gamma$ is not the only bridge (by assumption), there is some vertex $v' \in V(P)$ that is an attachment of some other (bigger) bridge $\Gamma'$. But now $P \cup \Gamma$ contains a path $P'$ that begins at $v$, avoids $v'$, and ends at one of the ends of $P$. Again, $P'$ contradicts our choice of $P$. This concludes the proof of Claim \ref{Claim: path P in C}.
\end{proof}

\begin{figure}[htp]
    \centering
    \includegraphics[scale = 0.6]{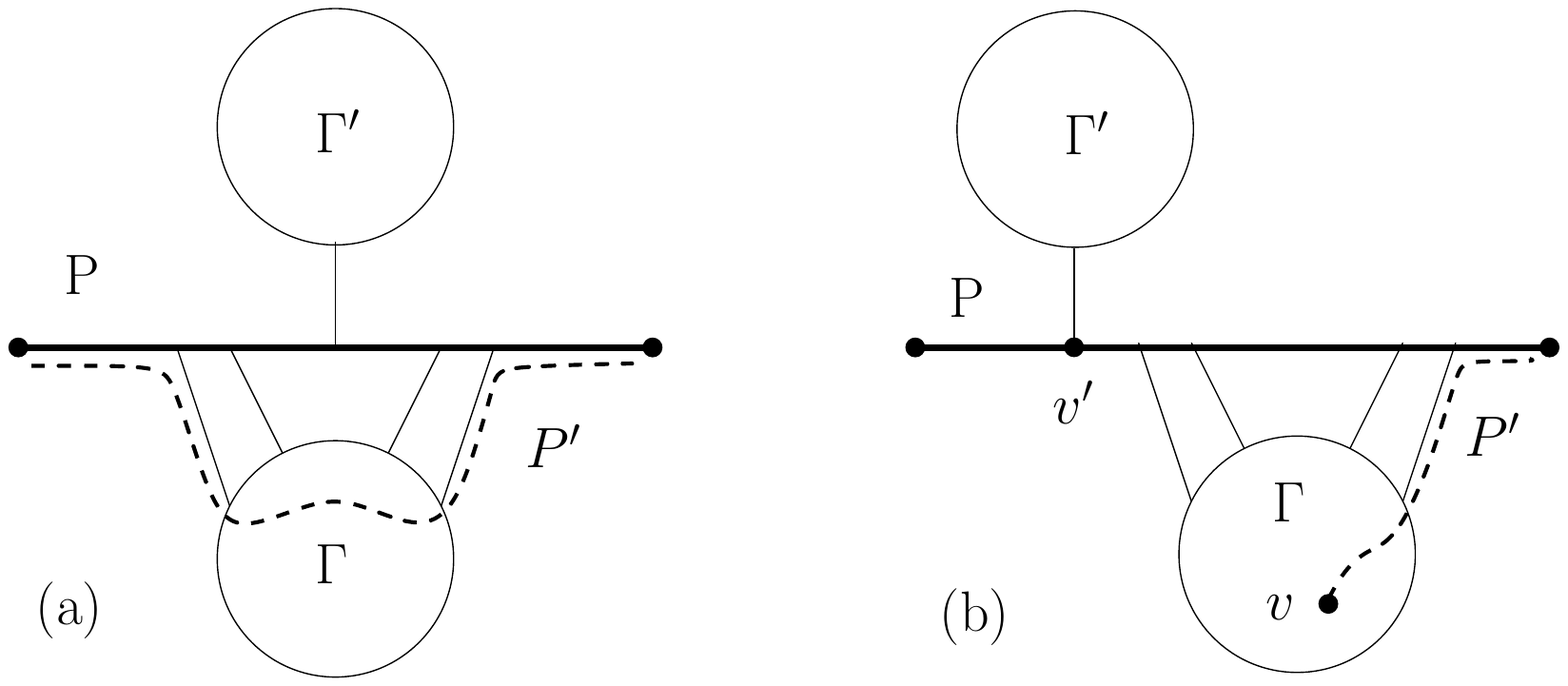}
    \caption{The path $P$ (bolded), with bridges $\Gamma$,$\Gamma'$ in $C$, and the path $P'$ (dashed) which improves $P$.}
    \label{fig:LexBridges}
\end{figure}

Let $P \subseteq C$ be a path as in Claim \ref{Claim: path P in C}. We will use $P$ to improve our partition.
Let $X$ be the end-edges of $P$. Construct $C'$ = $C\setminus E(P)$, $A' = A \cup X$, and $B' = B \cup E(P) \setminus X$.
We now verify that $A' \cup B' \cup C'$ is a partition which satisfies the five properties above. We examine each in order.
\begin{enumerate}[(a)]
    \item Observe that the endpoints of $P$ are in $A \cup B$, and $A' \cup B' \setminus A \cup B = E(P)$. Since $A \cup B$ is 2-connected, so is $A' \cup B'$.
    \item That $C'$ is connected follows from Claim \ref{Claim: path P in C}. And deleting the edges of a path between degree-1 vertices preserves the degree conditions.
    \item By Claim \ref{Claim: path P in C}, $C$ remains connected when the edges of $P$ are removed. Since the end-edges of $P$ are put into $A$, and since the internal vertices of $P$ have an incident edge in $C\setminus E(P)$, it means $A' \cup C'$ is a connected graph on the same vertex set as $A \cup C$. Since $A \cup C$ contains a spanning tree of $G$, so does $A' \cup C'$.
    \item We must verify that $X \subseteq \langle B' \rangle_2$. Let $x,y$ be the end-vertices of $P$. We know $\langle B \rangle _2$ is 2-connected, $x,y \in V(\langle B \rangle _2)$, the edges $E(P)\setminus X$ are in $B'$, and $B$ contains a cycle $D$ which is negative if $G$ is unbalanced. By 2-connectivity of $\langle B \rangle_2$, there exist two vertex-disjoint paths, say $P_x$ and $P_y \subseteq \langle B \rangle_2$, from $D$ to $\{x,y\}$. But now $Y = P \cup P_x \cup P_y \cup D$ has a positive cycle containing $X$, and $E(Y) \setminus X \subseteq \langle B' \rangle _2$. But this means $X \in \langle B' \rangle _2$.
    \item The cycle that is in $B$ is also in $B'$.
\end{enumerate}

Because $|C'| < |C|$, this means $A' \cup B' \cup C'$ contradicts our choice of $A \cup B \cup C$. Therefore $C$ is empty, which completes the proof.
\end{proof}

Our next goal is Theorem \ref{thm: decomposition to connected base and almost 2-base}, which is another extension of Seymour's decomposition to signed graphs.
However, first we require the following theorem, known as the 2-linkage theorem (actually will use Corollary~\ref{cor:to 2-link for cubic 3conn}), and Proposition~\ref{prop: peripheral cycle leaving G unbalanced} which is about the existence of a positive or negative peripheral cycle of a given sign disjoint from some negative cycle.

\begin{theorem}[Seymour \cite{SeymourP.D.1980Dpig}, Thomassen \cite{ThomassenCarsten19802G}]\label{thm: 2-linkage}
Let G be a 2-connected graph 
with vertices $x_1,x_2,y_1,y_2$. Then $G$ contains two edge-disjoint paths connecting $x_1$ with $x_2$ and $y_1$ with $y_2$, respectively, unless by a sequence of edge-contractions a graph $G'$ can be obtained from $G$, such that $x_1,x_2,y_1,y_2$ are contracted into $x'_1,x'_2,y'_1,y'_2$ respectively, and where $G'$ is 
\begin{enumerate}
    \item a 4-cycle with vertices $x'_1,y_1',x_2',y'_2$ in that cyclic order; or
    \item obtained from a 2-connected planar cubic graph by selecting a facial cycle and inserting the vertices $x'_1,y_1',x_2',y'_2$ in that cyclic order on the edges of that cycle.
\end{enumerate}
\end{theorem}

\begin{corollary}\label{cor:to 2-link for cubic 3conn}
Let $\Gamma$ be a 3-connected cubic signed graph where there does not exist $X \subseteq V(\Gamma)$ with $|X| \geq 2$, $|\delta(X)| = 3$, and $\Gamma[X]$ balanced. Let $G$ be a $2$-connected, balanced, induced subgraph of $\Gamma$, and let $x_1,x_2,y_1,y_2$ be vertices of degree $2$ in $G$. Then both of the following are true.
\begin{enumerate}
    \item\label{2-link cor1} If $x_1,x_2,y_1,y_2$ are the only vertices of degree $2$ in $G$, then $G$ contains two edge-disjoint paths connecting $x_1$ with $x_2$ and $y_1$ with $y_2$, respectively, unless $G$ can be embedded in the plane so that $x_1,y_1,x_2,y_2$ occur in that cyclic order on the edges of some facial cycle.
    \item\label{2-link cor2} If $y_3$ is an additional vertex of degree $2$ in $G$, then $G$ contains two disjoint paths $P_x,P_y$ so that the ends of $P_x$ are $x_1,x_2$, and the ends of $P_y$ are in the set $\{y_1,y_2,y_3\}$.
\end{enumerate}
\end{corollary}
\begin{proof}
First we prove 1. By Theorem~\ref{thm: 2-linkage}, either the desired paths exist, or by a sequence of contractions, a signed graph $G'$ can be obtained from $G$ such that $x_1,x_2,y_1,y_2$ are contracted into $x'_1,x'_2,y'_1,y'_2$ and $G'$ is one of the two types listed in that theorem. We will show that actually $G' = G$, i.e., the sequence of contractions is empty.
Let $v \in V(G')$ and let $A$ be the set of vertices of $G$ that were contracted into $v$. It is sufficient to show that $|A| = 1$. If $v \in \{x'_1,x'_2,y'_1,y'_2\}$, then $\delta(v)$ is a 2-edge-cut in $G'$. Further, $|x_1,x_2,y_1,y_2| \cap A = 1$ because these four vertices of degree 2 in $G$ are each contracted into a separate vertex. This means $\delta(A)$ is actually a 3-edge-cut in $\Gamma$, since $\deg_{G'}(v) = 2$ and there is only one vertex in $A$ with degree $2$ in $G$ (whose extra incident edge in $\Gamma$ contributes to the edge-cut). Hence by the assumptions of the corollary, $|A| = 1$. 
Otherwise, if $v \neq x_1',x'_2,y'_1,y'_2$, then $|\delta_{\Gamma}(v)| = 3$, and $x_1,x_2,y_1,y_2 \notin A$. Again, $\delta(A)$ is a 3-edge-cut in $\Gamma$, and by the assumptions of the corollary, $|A| = 1$. 

Now we prove 2. Suppose for contradiction the paths $P_x,P_y$ do not exist. Then by item 1 of this corollary there exists a cycle $C$ with vertices $x_1,y_1,x_2,y_2$ occurring in that cyclic order. But since $G$ is connected, there exists a (possibly trivial) path $P$ joining $y_3$ and $C$. The end of $P$ on $C$ is not $x_1$ or $x_2$ because they have degree 2. But this means $C \cup P$ contains the desired paths $P_x,P_y$, a contradiction.
\end{proof}

\begin{proposition}\label{prop: peripheral cycle leaving G unbalanced}
Let $G$ be a 3-connected, cubic signed graph.
\begin{enumerate}
    \item $G$ has a peripheral cycle $D$ so that $G-E(D)$ is unbalanced if and only if there exist disjoint cycles $C,N$ in $G$ with $N$ negative.
    \item $G$ has a negative peripheral cycle $D$ so that $G-E(D)$ is unbalanced if and only if there exist disjoint cycles $C,N$ in $G$, with both $C$ and $N$ negative.
\end{enumerate}
\end{proposition}
\begin{proof}
It is easy to see the `only if' direction is true in both statements. We must show the `if' direction is also true. Let $C,N$ be disjoint cycles in $G$ where $N$ is negative. As before (following Tutte), we call a connected component $B$ of $G-E(C)$ a \emph{bridge} of $C$, and say $V(C) \cap V(B)$ are its \emph{attachment vertices}. 

Choose $C$ to maximize the size (vertices plus edges) of the bridge $\Gamma$ containing $N$. Suppose, for contradiction, that there are at least two bridges of $C$. Because $G$ is 3-connected, there must be some bridge $\Gamma'\neq \Gamma$, and vertices $x,x',y,y'$ which occur in that cyclic order on $C$, so that $x,y$ are attachment vertices of $\Gamma$ and $x',y'$ are attachment vertices of $\Gamma'$. But then rerouting $C$ through $\Gamma'$ increases the size of $\Gamma$, a contradiction. Therefore $C$ has only one bridge, meaning it is a peripheral cycle. Moreover, since $N$ is negative, $G-E(C)$ is unbalanced, and statement 1 is true. To prove statement 2, we follow the same argument, but choose $C$ to be negative. And we observe that we can always do the rerouting so that the rerouted cycle $C'$ is negative. Indeed, if $C'$ is not negative, then $C \triangle C'$ is. Therefore $C$ is a negative peripheral cycle disjoint from $N$, as desired. 
\end{proof}

Let $H_n$ be a connected signed graph formed from the union of a negative cycle of length $n$ and $n$ isolated vertices by adding $n$ edges so that every vertex of the negative cycle has degree $3$, depicted in Figure~\ref{fig: NegativeSun}. By a \emph{negative sun} we mean any member of the family $\{H_n: n\geq 3\}$. These graphs will be required for the proof of the next theorem, as well as in Section \ref{subsect: two disj neg cycles}.  In the following proof we will also use the term \emph{$k$-separation} of a signed graph $G$ to mean a pair ($G_1,G_2$) of subgraphs so that $E(G_1) \cap E(G_2) = \emptyset$, $G_1 \cup G_2 =G$, and $|V(G_1) \cap V(G_2)| = k$. If $V(G_1) \setminus V(G_2) = \emptyset$, then the separation is \emph{trivial}. Finally, note that the following theorem is where we begin to use the notation of \emph{base}, as introduced earlier.

\begin{figure}
    \centering
    \includegraphics[scale = 0.75]{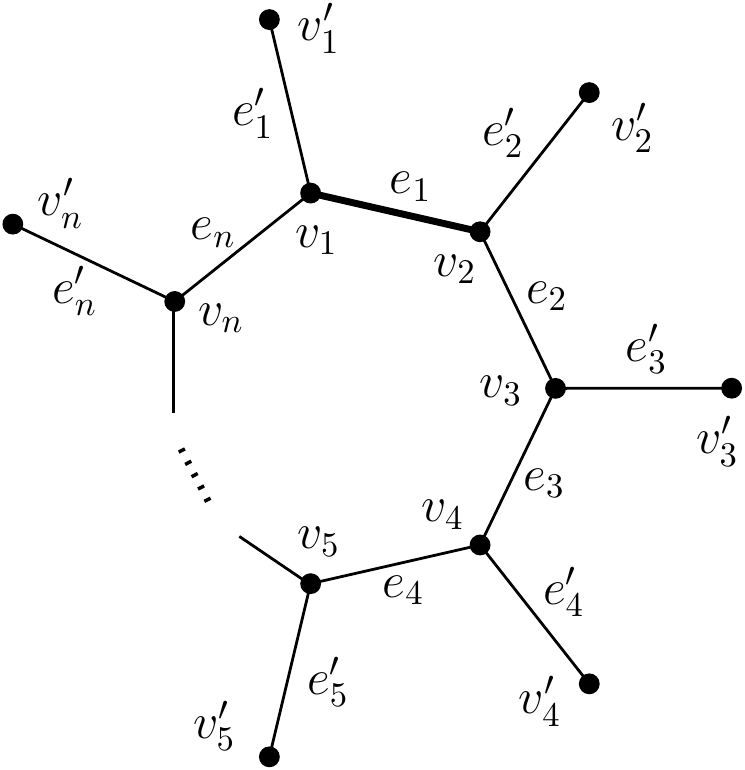}
    \caption{A \emph{negative sun} $H_n$. In this picture, the only negative edge is the bolded edge $e_1$; this particular labelling will be used in Section \ref{subsect: two disj neg cycles}.} 
    \label{fig: NegativeSun}
\end{figure}

\begin{theorem}\label{thm: decomposition to connected base and almost 2-base}
Let $G$ be a 3-connected, cubic signed graph that contains two disjoint negative cylces. Suppose there does not exist $X \subseteq V(G)$ such that $G[X]$ is balanced and
\begin{enumerate}[(i)]
    \item\label{assum: no 3ec separating balanced} $|X| \geq 2$ and $|\delta(X)| = 3$, or
    \item\label{assum: no planar 4-cut} $|X| \geq 3$, $|\delta(X)| = 4$, and $G[X]$ can be embedded in the plane such that all degree-2 vertices are incident to the unbounded face.
\end{enumerate}
Then there exists a partition $X_1\cup X_2$ of $E(G)$ such that $X_1$ is a connected base, $\langle X_2 \rangle_2 = E(G) - F$, where $F$ is the edge-set of a negative sun, $\langle X_2 \rangle_2$ is $2$-connected, and $X_2$ is unbalanced.
\end{theorem}

\begin{proof} 
\newcounter{count} Suppose $G$ is a counterexample to the theorem. We claim there exists (similar to Theorem \ref{thm: decomposition into tree and 2-base}, with differences italicized below) a partition with possibly empty parts $A \cup B \cup C$ of $E(G)$ so that 
\begin{enumerate}[(a)]
    \item\label{property2: AB 2-con} $A \cup B$ is a 2-connected graph;
    \item\label{prop2: C} $C$ is connected, \emph{unbalanced}, and all vertices of $C$ have degree 1 or 3 in $C$;
    \item\label{prop2: connected base} $A \cup C$ contains a \emph{connected base} of $G$;
    \item\label{prop2: 2-base} The 2-closure of $B$ contains $A$.
    \item\label{property2: -cyc in B} $B$ contains a \emph{negative} cycle.
\end{enumerate}
To show such a partition exists, choose for example a negative peripheral cycle $D$ so that $G - E(D)$ is unbalanced by Proposition~\ref{prop: peripheral cycle leaving G unbalanced}. Then $A = \emptyset, B = E(D),$ and $C = E(G) \setminus E(D)$ is a partition satisfying the requirements. It follows that we may choose $A \cup B \cup C$ with $C$ minimal, and so that $C \neq E(G)$.

We proceed with a series of claims.

\setcounter{claim}{0}

\begin{claim}\label{clm: not H_n}
$C \neq H_n$ for any $n \geq 3$.
\end{claim}
\begin{proofc}
Suppose for contradiction $C=H_n$ for some $n\geq 3$. We will show that $G$ satisfies the theorem. Let $X_1$ be a connected base in $A \cup C$, which we know exists by property~\ref{prop2: connected base}. We can choose $X_1$ so that $C \subseteq X_1$ (start with $C$, and add as many edges from $A$ as possible without creating another cycle). Let $X_2 = E(G) \setminus X_1$. Since $B\subseteq X_2$, it follows that $A\subseteq \langle X_2 \rangle_2$ by property \ref{prop2: 2-base}. This implies that $\langle X_2 \rangle_2 \supseteq E(G) \setminus C$, but in fact we get equality in this last containment, because every path between degree one vertices in $C=H_n$ has length at least 3, so the edges in $C$ can't be in the 2-closure. Therefore $\langle X_2 \rangle_2=A\cup B$, and it is $2$-connected by property \ref{property2: AB 2-con}. Finally, $X_2$ is unbalanced by property \ref{property2: -cyc in B}. Thus the partition $X_1\cup X_2$ satisfies the theorem, contradicting that $G$ is a counterexample.
\end{proofc}

\begin{claim}\label{clm: disj cyc and path}
 $C$ does not have a path $P$ between two vertices of degree one such that $C-E(P)$ is unbalanced. 
 \end{claim}
 \begin{proofc}
 Suppose for contradiction that such a $P$ exists. Then there is a negative cycle $N$ in $C - E(P)$. Choose $P$ to maximize the bridge containing $N$ (using again the term bridge from  Claim \ref{Claim: path P in C} in the proof of Theorem \ref{thm: decomposition into tree and 2-base}), and after that prioritize the lexicographic ordering of the sizes of the remaining bridges (as in Claim \ref{Claim: path P in C} of Theorem \ref{thm: decomposition into tree and 2-base}).  As in the proof of that past claim, this means $P$ has exactly one bridge which contains $N$, and is therefore unbalanced. But now we can improve our choice of partition. Let $X$ be the end-edges of $P$. Define $A' = A \cup X$, $B' = B \cup E(P) \setminus X$, and $C' = C \setminus E(P)$. The partition $A' \cup B' \cup C'$ satisfies Properties~\ref{property2: AB 2-con} to \ref{property2: -cyc in B}  using the same argument in the proof of Theorem \ref{thm: decomposition into tree and 2-base}, noting that \ref{prop2: C} and \ref{prop2: connected base} follow because $C - E(P)$ contains $N$. This contradicts that our partition was chosen to have $C$ minimal.
\end{proofc}

\begin{claim} \label{clm: 1-edge-cut}  Every $1$-edge-cut of $C$ has the form $\delta(v)$ for some vertex $v \in V(C)$. 
\end{claim}
\begin{proofc} Suppose for contradiction $\delta(A)$ is a $1$-edge-cut where $|A|$ and $|\overline{A}|$ are at least 2. Since $C$ is unbalanced, there is a negative cycle in, without loss of generality, $G[A]$. By $3$-connectivity of $G$, $|\delta_G(\overline{A})| \geq 3$, and so $C[\overline{A}]$ has at least two vertices of degree one. But then there is a path joining two of those vertices in $C[\overline{A}]$ which is disjoint from a negative cycle in $C[A]$, contradicting Claim \ref{clm: disj cyc and path}. 
\end{proofc}

\begin{claim} \label{clm: No subgraph C'} There does not exist a 2-separation $(G_1,G_2)$ of $C$ so that $G_2$ is balanced and contains a cycle.
\end{claim}
\begin{proofc} Suppose for contradiction that such a 2-separation $(G_1,G_2)$ exists. Let $\{u_1,u_2\} = V(G_1) \cap V(G_2)$. 
We may assume $G_2$ has minimal number of vertices. This means, in particular, that $\deg_{G_2}(u_1) = \deg_{G_2}(u_2) = 2$, and that $G_2$ has no vertex of degree 1 adjacent to $u_1$ or $u_2$.

Since $C$ is unbalanced,
and every $1$-edge-cut in $C$ is trivial (by Claim \ref{clm: 1-edge-cut}),
there exist two disjoint (possibly degenerate) paths from $\{u_1,u_2\}$ to some negative cycle in $C$. Since $G_2$ is balanced, these paths are contained in $G_1$, and therefore any $u_1,u_2$-path in $G_2$ can be combined with some path in $G_1$ to make a negative cycle. This means there cannot exist two disjoint paths $P_1,P_2 \subseteq G_2$ where the ends of $P_1$ are $u_1,u_2$ and both ends of $P_2$ have degree 1 in $G_2$, because this would contradict Claim \ref{clm: disj cyc and path}.

If there are at least three vertices of degree $1$ in $G_2$, then we claim such disjoint paths exist by Corollary~\ref{cor:to 2-link for cubic 3conn}~(\ref{2-link cor2}). Indeed, for $x_1,x_2,y_1,y_2,y_3$ in that corollary, take $x_1,x_2$ to be $u_1,u_2$, and $y_1,y_2,y_3$ to be the neighbors of the degree-1 vertices, noting that each $y_i$ is distinct or we obtain a non-trivial~1-edge-cut in $C$. Hence we can apply Corollary~\ref{cor:to 2-link for cubic 3conn}~(\ref{2-link cor2}) to $G_2$, after deleting its vertices of degree 1,  and the result contradicts  Claim \ref{clm: disj cyc and path}. Therefore $G_2$ has at most two vertices of degree 1.

If there is an edge $e$ with ends $u_1$ and $u_2$ in $E(G_1)$, then $C = G_2 + e$, since $\deg_{G_2}(u_1) = \deg_{G_2}(u_2) = 2$ and $G$ cubic. This means, in particular, that all of the degree 1 vertices of $C$ are in $G_2$. But $G_2$ has at most two such vertices, and deleting them contradicts the $3$-connectivity of $G$. So $u_1u_2 \not\in E(G_1)$. Let $G_2'$ be $G_2$ with vertices of degree-$1$ deleted. Then $G_2'$ is induced in $G$. 

Since $G_2'$ is balanced and has at least two vertices, assumption~\ref{assum: no 3ec separating balanced} says that there cannot be exactly three edges out of $G_2'$ in $G$. However $u_1, u_2$ contribute two such edges, and any vertex of degree one in $G_2$ contributes one. So there must be exactly two vertices of degree-1 in $G_2$. Say their neighbors are $u_3,u_4 \in V(G_2')$.

Now apply Corollary~\ref{cor:to 2-link for cubic 3conn}~(\ref{2-link cor1}) to $G_2'$. Since paths $P_1,P_2$ as above do not exist in $G_2$, we get that $G_2'$ can be embedded in the plane so that $u_1,u_2,u_3,u_4$ are incident to the unbounded face. But this contradicts assumption \ref{assum: no planar 4-cut}, because $G_2$ is balanced, planar, has at least three vertices, and induces a 4-edge-cut in $G$  (with one edge corresponding to each $u_i$).
\end{proofc}

In the following claim we use the term \emph{bond}, which is a minimal non-empty edge-cut (meaning both sides of the cut are connected).

\begin{claim}\label{clm: 2edge bond}
 If $\delta(X)$ is a 2-edge-bond in $C$, then either $C[X]$ or $C[\overline{X}]$ is $K_2$. 
 \end{claim}

 \begin{proofc}
 If the two edges of the bond are adjacent then indeed one side is $K_2$ by Claim \ref{clm: 1-edge-cut}. So we may assume that there are two distinct vertices incident to the bond on each side, and there are paths joining both these pairs in $C - \delta(X)$, the internal vertices of which must have degree three. If both $C[X]$ and $C[\overline{X}]$ are trees, then $C$ is $H_n$ for some $n \geq 3$ (again since $C$ has no non-trivial 1-edge-cuts by Claim~\ref{clm: 1-edge-cut}).  
Hence, by Claim~\ref{clm: not H_n}, $C[X]$ contains a cycle. This means $C[X]$ is unbalanced by Claim \ref{clm: No subgraph C'}. If $C[\overline{X}]$ has two vertices of degree~$1$, then a path between those two vertices in $C[\overline{X}]$ contradicts Claim \ref{clm: disj cyc and path} because $C[X]$ is unbalanced. So it must be that $C[\overline{X}]$ has at most one vertex of degree~$1$. But then by assumption \ref{assum: no 3ec separating balanced}, $C[\overline{X}]$ is unbalanced. However, by 3-connectivity of $G$, there must be at least $3$ vertices of degree~$1$ in $C$. So it must be that $C[X]$ has two vertices of degree 1. But this contradicts Claim \ref{clm: disj cyc and path} because there is a path between those two vertices of degree~$1$ in $C[X]$ which avoids a negative cycle in $C[\overline{X}]$. 
\end{proofc}

With this last claim in place, we proceed to complete the proof. 

By Claims \ref{clm: 1-edge-cut} and \ref{clm: 2edge bond}, $C$ is very close to 3-edge-connected, and we can construct a 3-edge-connected, cubic (and so 3-connected) signed graph $C'$ from $C$ by deleting vertices of degree 1 and suppressing the resulting vertices of degree 2 (with each new edge taking the sign of the associated 2-path). Note that none of the suppressed vertices were adjacent in $C$, otherwise there would be a 2-edge-bond with a path of length 3 on one side and $K_2$ on the other side (by Claim~\ref{clm: 2edge bond}) and so $C=H_3$ (which contradicts Claim~\ref{clm: not H_n}). So if we let $S$ be the set of new edges formed via suppression, there is a natural bijection between $S$ and the vertices of degree~$1$ in $C$. Since $G$ is 3-edge-connected, $C$ has at least three vertices of degree one, and so $|S|\geq 3$.

Since $C$ has a negative cycle, so does $C'$. And we may choose a negative peripheral cycle $D \subseteq C'$ by Corollary \ref{cor: neg peiph cyc}. If $|E(S) \setminus E(D)| \geq 2$, then since $D$ is peripheral there is a path containing two of those $S$-edges which avoids $E(D)$, and contradicts Claim \ref{clm: disj cyc and path}. So there is at most one edge in $S$ that is not in $E(D)$. Because $|S|\geq 3$, this means $|E(D) \cap S| \geq 2$. 

If $C' - E(D)$ has a negative cycle, it is disjoint from a path between two edges in $E(D) \cap S$, which would again contradict Claim \ref{clm: disj cyc and path}. Hence $C' - E(D)$ is balanced. Therefore we may assume all the negative edges in $C'$ are on $D$. If this signature has more than one negative edge, it has at least three because $D$ is negative. But if there are three negative edges on $D$, then there is a path $Q$ on $D$ containing two edges of $S$ which avoids one of the negative edges. And that negative edge makes a negative cycle in $C' - E(Q)$. This contradicts Claim \ref{clm: disj cyc and path}. Hence $C'$ is 1-unbalanced.

Since $C'$ is 1-unbalanced, so is $C$.
Let $e = uv$ be the negative edge in $C$. There is a trivial 2-separation $(C-e,C[\{u,v\}])$ in $C$. Since $C-e$ is balanced, it must be acyclic by Claim~\ref{clm: No subgraph C'}. But now, repeating the same argument at the start of Claim \ref{clm: 2edge bond}, we get that $C=H_n$ for some $n \geq 3$. This final contradiction to Claim~\ref{clm: not H_n} completes our proof.
\end{proof}

The following corollary is a generalization of the decomposition found by Jaeger, Linial, Payan and Tarsi \cite{JaegerFrancois1992Gcog}; it may be of independent interest, with potential applications to other problems on signed graphs. A \emph{degenerate negative sun} is a negative sun $H_n$ where the vertices of degree one may not be distinct. A signed graph is \emph{cyclically $k$-edge-connected} if at least $k$ edges must be removed in order to disconnect it into two components that each contain a cycle.

\begin{corollary}\label{cor: decomp general JLPT}
    Let $G$ be a cyclically 4-edge-connected, cubic signed graph with no positive cycle of length at most 5. Then there exists a partition $X_1\cup X_2$ of $E(G)$ such that $X_1$ contains a connected base and $\langle X_2 \rangle_2 = E(G) - F$, where $F$ is either empty or the edge-set of a degenerate negative sun.
\end{corollary}

\begin{proof} If $G$ is balanced, then the corollary follows from Theorem \ref{thm: decomposition into tree and 2-base}, where $F = \emptyset$. So we may assume $G$ is unbalanced.

    If $G$ has no two disjoint cycles, then we can choose a negative peripheral cycle $D$ (by Corollary~\ref{cor: neg peiph cyc}), and we'll get that $G-E(D)$ is acyclic. Let $S = V(G)\setminus V(D)$. Suppose for contradiction that there is a vertex $v \in S$ so that $N_G(v) \subseteq S$. Since $G$ is 2-connected each $w\in N_G(v)$ has two paths to $D$, all of which are disjoint (except for $w$) because $G$ is cubic and $G-E(D)$ is acyclic. But now it is straightforward to verify (with some cases) that these paths together with $D$, $v$, and $\delta_G(v)$ contain a negative cycle disjoint from some other cycle, a contradiction. 
    Thus every $v \in S$ is adjacent to some vertex on $D$, and our desired partition is as follows: let $X_1$ be $E(D)$ plus one edge connecting each vertex in $S$ to $D$, and $X_2$ be the remaining edges. So we may assume $G$ has two disjoint cycles.
    
    If $G$ has two disjoint negative cycles, then the corollary follows from Theorem \ref{thm: decomposition to connected base and almost 2-base}: (i) follows from cyclic 4-edge-connectivity and (ii) follows from Euler's formula because any balanced $2$-connected subcubic planar graph with at most four vertices of degree two has a positive cycle of length at most five and $G$ has no such cycle. So we may assume $G$ has no two disjoint negative cycles.

    Finally, it must be that $G$ is unblanaced with two disjoint cycles, but no two disjoint negative cycles. 
    Now we claim that the corollary follows from the proof of Theorem \ref{thm: decomposition to connected base and almost 2-base}. Specifically, we claim that the conclusion of Theorem \ref{thm: decomposition to connected base and almost 2-base}, with the last part (that $X_2$ is unbalanced) removed, is true for our $G$. The proof is almost exactly the same, with the following important difference: we ignore property (e) in the proof of that theorem. By Proposition \ref{prop: peripheral cycle leaving G unbalanced} (1), $G$ has a positive peripheral cycle $D$ so that $G-V(D)$ is unbalanced. The partition $A \cup B \cup C$ where $A = \emptyset$, $B = E(D)$, and $C = E(G)\setminus E(D)$ is a partition satisfying the requirements of items (a) to (d). Apart from this, the proof proceeds as written. This works because, apart from the edges of the negative sun we will find, the remainder of the graph $G$ is balanced. Specifically, this means the $2$-closure of $B$ is equal to its $2$-closure as an unsigned graph.     
\end{proof}

\section{Proof of Theorem \ref{main}}\label{sect:proof of main}

Let $G$ be a $3$-edge-connected, 2-unbalanced signed graph. In order to prove Theorem~\ref{main} we must show that $G$ is $A$-connected for every abelian group $A$ with $|A|\geq 6$ and $|A|\neq 7$. Suppose that $G$ is a minimum counterexample in the sense of Lemma~\ref{lem:reduction} (ie. with respect to $\sum_{v \in V(G)}|\deg(v)-3|$, and subject to that with respect to $|V(G)|$). Then, in particular, $G$ is cubic and 3-connected. This means we can apply Theorem \ref{thm: decomposition into tree and 2-base}, which decomposes the graph into a spanning tree and a $2$-base; in Section 6.1 we will how to leverage this decomposition to show $A$-connectivity when $|A|$ is composite.

Recall that earlier in this paper, in Theorem \ref{thm: no 2 disj negative cycles}, we showed that a minimum counterexample to Theorem \ref{main} must contain two disjoint negative cycles. With this additional assumption, we can apply Theorem \ref{thm: decomposition to connected base and almost 2-base}, which decomposes a graph into a connected base, containing a negative sun, and a set of edges that almost form a $2$-base (missing exactly the edges of the negative sun). In Section 6.2 we will use this to complete our proof of Theorem \ref{main} in the remaining case, namely when  $|A|$ is prime, $|A|\geq 11$, and $G$ does have two disjoint negative cycles.

Note that in what follows, flow-values should be assumed to be zero on edges where they are not defined.

\subsection{The case for abelian groups of composite order.}\label{subsect:composite}

\begin{theorem}\label{thm: Z2 x Z3}
Let $G$ be a $3$-edge-connected, 2-unbalanced signed graph, and let $A$ be an abelian group with $|A|\geq 6$. Suppose $G$ is not $A$-connected and suppose that $G$ is chosen to be minimum with respect to $\sum_{v \in V(G)}|\deg(v)-3|$, and subject to that with respect to $|V(G)|$. Then $|A|$ must be prime.
\end{theorem}

\begin{proof}[Proof of Theorem \ref{thm: Z2 x Z3}]
We suppose that $|A|$ is composite and will show that $G$ is $A$-connected. By Lemma \ref{lem:reduction}, $G$ is cubic and 3-connected, and therefore by Theorem \ref{thm: decomposition into tree and 2-base} has an edge-partition $T \cup B$ where $T$ is a spanning tree and $B$ is a 2-base.

Let $\overline{f}: E(G) \rightarrow A$ be a set of forbidden edge values, corresponding to some initial orientation of $G$. We will use the partition $T \cup B$ to find a flow $\phi: E(G) \rightarrow A$ so that $\phi(e) \neq \overline{f}(e)$ for all $e \in E(G)$ (with either the same orientation or with $\overline{f}$ appropriately changed). By Proposition \ref{equivAcon}, this implies that $G$ is $A$-connected.

Recall that the converse of Lagrange's Theorem is true for abelian groups, i.e. for every divisor $d$ of $|A|$ there is a subgroup of $A$ of order $d$. Since $|A|$ is composite there exists a proper, nontrivial subgroup $N$ and quotient group $A/N$; choose $|N|$ as small as possible. Since $|A|\geq 6$, this means $|A/N|\geq 3$. (In particular, if $A=\mathbb{Z}_{2k}$, we will choose $N=\{0, k\}$ and $A/N=\{ \{0,k\}, \{1,k+1\}, \ldots, \{k-1, 2k-1\} \}=\{0+N, 1+N, \ldots, k-1+N\}$).

The first of two main steps in this proof is to construct a flow $\phi_1: E(G) \rightarrow A/N$ with the property that $\phi_1(e)\not\in \overline{f}(e)+N$ for all $e\in T$. (For example, with the $A/N$ mentioned above, if $\overline{f}(e)=3$, then we want $\phi_1(e)\not\in \{3, 3+k\}$). To this end, since $B$ is a 2-base, there is a list of positive cycles $C_1, \dots, C_t$ with the following two properties:
\begin{enumerate}
    \item for each $1\le i \le t$, $|W_i := C_i \setminus (\cup_{j=1}^{i-1} C_j \cup B)| \le 2$; and
    \item $T \subseteq \cup_{i=1}^t W_i$.
\end{enumerate}

Working from $C_t$ to $C_1$, we will choose a sequence of flows $f_i: E(C_i) \rightarrow A/N$ for all $t\geq i\geq 1$, so that $\phi_1 := \sum_{i=1}^t f_i$  has $\phi_1(e) \not\in \overline{f}(e) + N$ for all $e \in T$. When choosing the flow $f_i$, note that $W_i \cap E(C_j) = \emptyset$ for any $j < i$. And so in choosing the flow $f_i$, we are setting the value $\phi_1(e)$ for each $e \in W_i$. 
Since $|A/N| \geq 3$, there are least $3$ choices for $f_i$, and because $|W_i| \le 2$ there are at most two values of flow $f_i$ which would cause $\phi_1(e) \in \overline{f}(e) + N$ for some $e \in W_i$. We can therefore choose $f_i$ so that $\phi_1(e) \not\in \overline{f}(e) + N$ for every $e \in W_i$. After defining all of $f_t, \ldots, f_1$, we see that $\phi_1$ is as desired, since $T\cap B=\emptyset$. From $\phi_1$, obtain $\phi_1': E(G) \rightarrow A$ by setting $\phi_1'(e)$ to be the minimal element in the coset of $\phi_1(e)$.

The second main step in this proof is to construct a flow $\phi_2: E(G) \rightarrow N$ so that $\phi_2(e) \neq \overline{f}(e) - \phi'_1(e)$ for all $e \in B$. We consider two cases based on the parity of $|A|$.

Recall that if $|A|$ is even, then we chose $|N|= 2$. Here, for each $e \in B$ let $C_e$ be the fundamental cycle of $e$ with $T$. We want to define $f_e: E(C_e) \rightarrow N$, and if $C_e$ is a positive cycle it's clear there are exactly two choices for this definition, one per element in $N$. However since $|N|=2$, we also get these two choices even if $C_e$ is a negative cycle. So, in either case, we can define $f_e$ so that $f_e(e) \neq \overline{f}(e) - \phi'_1(e)$. Then by setting $\phi_2 := \sum_{e \in B} f_e$, we get our desired result.

If $|A|$ is odd, then $|N| \geq 3$. Because $G$ is 2-unbalanced there exist $b, b' \in B$ so that the fundamental cycles of both $b$ and $b'$ with $T$ are negative. (Because $T$ is balanced, we may choose a signature of $G$ so that every edge in $T$ is positive.) Let $T' = T \cup \{b'\}$, which is a connected base. For each $e \in B \setminus \{b,b'\}$, there exists $D_e \subseteq T' \cup \{e\}$ such that $D_e$ contains $e$ and $D_e$ is either a positive cycle or a barbell. We can choose a flow $g_e: D_e \rightarrow N$ so that $g_e(e) \neq \overline{f}(e) - \phi'_1(e)$, because $|N|\geq 3$ and there is only one value on $e$ we need to avoid. Then $\phi_2':=\sum_{e\in B \setminus \{b,b'\}} g_e$ has the property we want for $\phi_2$, except possibly on the edges $b, b'$. To fix these two edges, note that there exists $D_b\subseteq T' \cup \{b\}$ so that $D_b$  contains both of $b,b'$, so that $D_b$ is either a positive cycle or a barbell, and so that if $D_b$ is a barbell, then neither of $b,b'$ is in the path of that barbell. Because  $|N| \geq 3$, and we need only fix two edges, this means there is is a flow $g_{b}:D_{b} \rightarrow N$ that fixes both $b$ and $b'$ simultaneously. That is, we achieve $g_{b}(b) \neq \overline{f}(b) - \phi'_1(b)$ (note that since $b\not\in T'$ it was not given a value other than zero by $\phi_2'$) and $g_{b}(b') \neq \overline{f}(b') - \phi'_2(b') - \phi'_1(b')$. Finally, we set $\phi_2 := \phi_2' + g_b$.

But now our theorem is satisfied by the $A$-flow $\phi := \phi_1' + \phi_2$. To see this, note that every $e\in B$ has $\phi(e) \neq \overline{f}(e)$ by the condition on $\phi_2$. But the same follows for all $e \in T$ by the condition on $\phi_1$ because $\phi_2(e) \in N$.
\end{proof}

\subsection{The case for groups of prime order at least 11, when $G$ has two disjoint negative cycles.}\label{subsect: two disj neg cycles}

The goal of this subsection is to complete the proof of  Theorem~\ref{main}.
Before doing so, we require the following result of Cheng, Lu, Luo, and Zhang, which converts a $\Z_2$-flow to an integer 3-flow.

\begin{theorem}[Cheng, Lu, Luo, Zhang \cite{ChengJian2018SGFM})] \label{lemma2to3} Let $G$ be a connected signed graph. If $G$ has a $\mathbb{Z}_2$-flow $f_1$ such that $\{e|f_1(e)\neq 0\}$ contains an even number of negative edges, then $G$ also has an integer-valued 3-flow $f_2$ with $f_2(e)\in\{1, -1\}$ for all $e$ with $f_1(e)\neq 0$. 
\end{theorem}

We also require the following technical lemma, which shows that we can always put a certain flow on a negative sun. Since our remaining case concerns groups of prime $p$ order, and all such groups are isomorphic to the cyclic group $\Z_p$, it suffices to focus on  $\Z_p$ specifically.

\begin{lemma}\label{lem:flow star on $H$}
Let $G$ be a signed graph with a negative sun $H \subseteq G$ with negative cycle $C\subseteq H$. Suppose $G-V(C)$ is $2$-connected, and there is a negative cycle $N \subseteq G$ so that $H$ and $N$ are edge-disjoint. 
Then for any prime $p \ge 11$, and any $\overline{f}: E(H)\rightarrow \Z_p$ there is a flow $f: E(G) \to \Z_p$ so that 
\begin{enumerate}[(i)]
    \item $f(e') \neq \overline{f}(e')$ for some $e' \in E(H)$, and
    \item $f(e) \notin Y(e)$ for all $e \in E(H) \setminus \{e'\}$, where $Y(e) := \{\overline{f}(e), \overline{f}(e) \pm 3, \overline{f}(e) \pm6\}$.
\end{enumerate}
\end{lemma}
\begin{proof}
Let the negative sun be the $H_n$ that is signed and labelled as in Figure \ref{fig: NegativeSun}.
Since $N$ is disjoint from $H_n$, and $G - C$ is $2$-connected, for each $e_i$ there is a positive cycle $D_i \subseteq G$ so that $E(D_i) \cap E(H_n) = \{e_i', e_i, e_{i+1}'\}$ (modulo $n$).

Begin by orienting the edges of $H_n$ as in Figure \ref{fig: flow on H when boundary 0 everywhere}, noticing that the orientation differs slightly depending on the parity of $|V(C)|$. First suppose that given this orientation, the boundary of $\overline{f}$ is equal to 0 for every vertex in $C$, that is, $\partial \overline{f}(v)=0$ for all $v\in C$.

If $|V(C)|$  is odd, then for $i \in \{1 \dots n\}$, let $f_i: E(D_i) \rightarrow \Z_p$ be a flow with the given orientation so that $f_i(e_i) = \overline{f}(e_i) + 1$, and set $f = \sum_{i=1}^n f_i$. Since $p> 7$, the flow $f$ satisfies our desired condition (ii) on the edges $e_1, e_2, \ldots, e_n$. On the other hand, given our assumption about the boundary, for any $i$,
$$f(e_i')= \overline{f}(e_{i-1})+\overline{f}(e_{i})+2=\overline{f}(e'_i)+2.$$
Hence, since $p>8$, these edges also satisfy $(ii)$.

If $|V(C)|$ is even, then construct $f$ almost as in the odd case, with the exception that we add one to the flow $f_n$, so that $f_n(e_n) = \overline{f}(e_n) + 2$. This has the effect that $f$ takes the same values as the odd case except for the three edges $e_n',e_n$ and $e_1'$. These have flow values equal to their forbidden value plus $3$, $2$, and $1$ respectively (note in the last case that $\overline{f}(e_1')+\overline{f}(e_1)=\overline{f}(e_n)$); 
see Figure \ref{fig: flow on H when boundary 0 everywhere}. Therefore, since $p>8$, 
$e_n'$ satisfies (i), 
and all other edges in $H_n$ satisfy (ii).

\begin{figure}
    \centering
    \includegraphics[scale = 1]{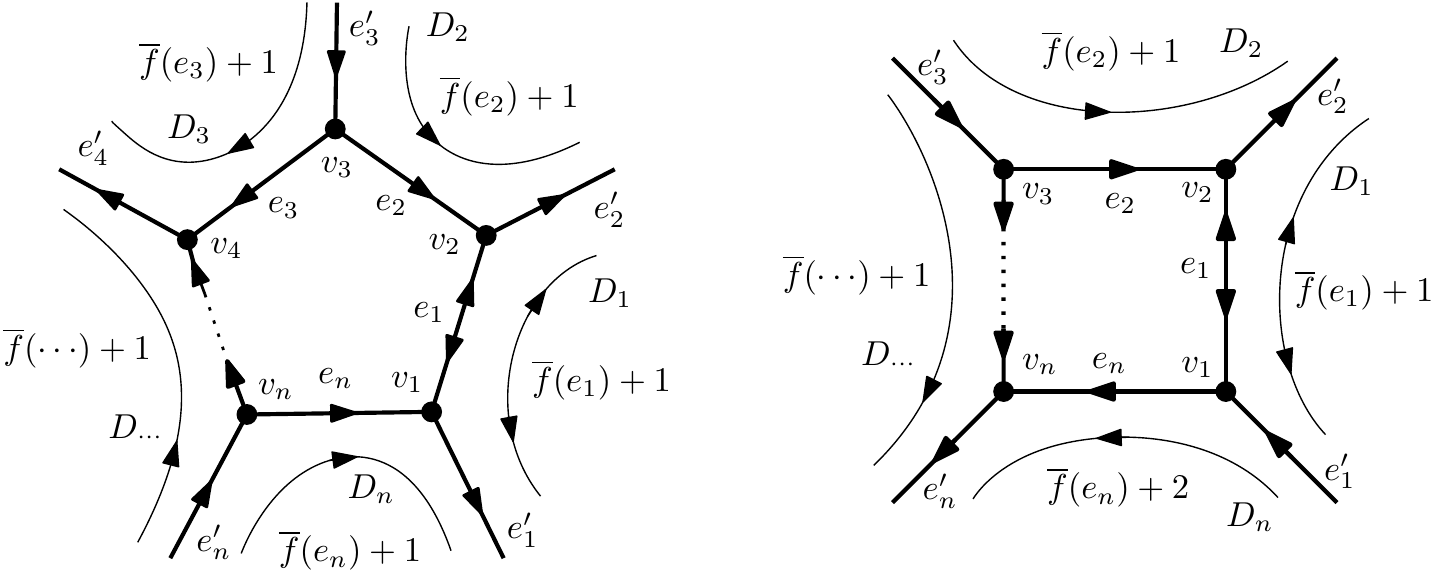}
    \caption{The graph $H_n$ when $|V(C)|$ is odd (left) and even (right). The flows which sum to $f$ are depicted for the case when the boundary of $\overline{f}$ is zero for all vertices in $V(C)$.}
    \label{fig: flow on H when boundary 0 everywhere}
\end{figure}

We may now assume that there is some $v \in V(C)$ such that $\partial \overline{f} (v) \neq 0$. Since the absolute value of $|\partial \overline{f} (v)|$ is independent of changes in orientation and signature (provided  $\overline{f}$ is modified appropriately), it is without loss to suppose that $v_2$ has nonzero boundary. Let $g_2: D_2 \rightarrow \Z_p$, be a flow so that $g_2(e_2) = \overline{f}(e_2') - \overline{f}(e_1)$.
Because $\partial\overline{f} (v_2) \neq 0$, it follows that $g_2(e_2) \neq \overline{f}(e_2)$.
We will see that the flow $g_2$ has the effect of matching up the edges $e_1$ and $e_2'$ so that they can be fixed (made to not equal any of their forbidden $Y$-values) simultaneously. 
But first, we fix most of the edges of $E(H_n)$ two at a time using the positive cycles $D_3 \dots D_{n}$.
We can fix two edges at a time because $p \geq 11$, and we are restricting 5 values per edge, for a maximum of $5 \cdot 2 = 10 < 11$ bad values.
We define $g_i : D_i \rightarrow \Z_p$ for $3\leq i\leq n$ to fix edges $e_i$ and $e_i'$. In the end, $g' := \sum_{i=2}^n g_i$ is such that $g'(e) \not\in Y(e)$ for all $e \in E(H_n)\setminus \{e_1,e_1', e_2,e_2'\}$ and $g'(e_2) \neq \overline{f}(e_2)$.
Finally, we use a flow $g_1 : D_1 \rightarrow \Z_p$ to fix the remaining three edges $e_1',e_1,e_2'$ so that $g = g' + g_1$ is such that $g(e) \not\in Y(e)$ for all $e \in \{e_1',e_1,e_2'\}$.
To see that these three edges can be fixed simultaneously, note that
$g'(e_1) = 0$ and $g'(e_2') = \overline{f}(e_2') - \overline{f}(e_1)$.
So if $g(e_1) = \overline{f}(e_1) + k$ for some $k \in \Z_p$, then $g(e_2') = \overline{f}(e_2') + k$.
This means the set of flows on $D_1$ that are bad for $e_1$ are exactly the flows that are bad for $e_2'$. So we again have at most ten bad flow values, but this time we can fix three edges. This means $g$ is as required: $g(e_2) \neq \overline{f}(e_2)$, and $g(e) \not\in Y(e)$ for all $e \in E(H_n) \setminus \{e_2\}$.
\end{proof}

Now we proceed with our proof of Theorem \ref{main}, restated here for convenience.

\setcounter{theorem}{1}
\begin{theorem} Let $G$ be a $3$-edge-connected, 2-unbalanced signed graph. Then $G$ is $A$-connected for every abelian group $A$ with $|A|\geq 6$ and $|A|\neq 7$.
\end{theorem}

\begin{proof}
Let $G$ be a counterexample that is minimum in the sense of Lemma \ref{lem:reduction}.  By Theorem \ref{thm: Z2 x Z3} we may assume that $|A|$ is prime and $|A|\geq 11$. By Theorem \ref{thm: no 2 disj negative cycles} we may assume that $G$ has two disjoint negative cycles. We claim that $G$ satisfies the assumptions of Theorem \ref{thm: decomposition to connected base and almost 2-base}. Certainly, by Lemma \ref{lem:reduction}, it is $3$-connected and cubic. Suppose that there is $X \subseteq V(G)$ with $G[X]$ balanced. Notice that if $|\delta(X)| = 3$, then by Lemma \ref{lem:reduction}, $|X|=1$ and similarly, if $|\delta(X)| = 4$ and $G[X]$ is planar, then by Lemma \ref{lem:reduction} $G[X]$ must be a path on two vertices.

We now apply Theorem \ref{thm: decomposition to connected base and almost 2-base} to get a partition $T \cup B = E(G)$ such that
\begin{enumerate}
    \item $T$ is a connected base,
    \item $\langle B \rangle_2 = E(G)-F$, where $F \subseteq T$ are the edges of a negative sun,
    \item $\langle B \rangle_2$ is $2$-connected, and
    \item there exists a negative cycle $N \subseteq B$.
\end{enumerate}

Let $\overline{f}: E(G) \rightarrow \Z_p$ be a set of forbidden edge values, corresponding to some initial orientation of $G$. We will use the partition $T \cup B$ to find a flow $\phi: E(G) \rightarrow \Z_p$ so that $\phi(e) \neq \overline{f}(e)$ for all $e \in E(G)$ (with either the same orientation or with $\overline{f}$ appropriately changed). By Proposition \ref{equivAcon}, this implies that $G$ is $Z_p$-connected.

As in the proof of Theorem \ref{thm: Z2 x Z3}, we will define two separate flows, each focusing on the edges of $T$ and then $B$ respectively, which we sum to achieve our desired result. In particular, for the first of these, we aim to construct a flow $\phi_1: E(G) \rightarrow \Z_p$, so that $\phi_1(e') \neq \overline{f}(e')$ for some $e' \in F$, and $\phi_1(e) \not \in Y(e) := \{\overline{f}(e), \overline{f}(e) \pm 3, \overline{f}(e) \pm 6\}$ for every $e \in T \setminus \{e'\}$. We begin by applying  
Lemma \ref{lem:flow star on $H$} to get a flow $f^*$ on $G$ satisfying these requirements for all $e\in F$. Note that we can apply Lemma \ref{lem:flow star on $H$} here because $\langle B \rangle_2$ is $2$-connected, $N$ is a negative cycle that is edge-disjoint from $F$, and $p$ is a prime $\geq 11$. To finish building $\phi_1$, we must ensure every remaining edge $e \in T' = T \setminus F$ does not receive a flow-value in the set $Y(e)$ of five forbidden values. We do this in a manner similar to the proof of Theorem \ref{thm: Z2 x Z3}.
As before, because $T' \subseteq \langle B \rangle_2$, there is a list of positive cycles $C_1, \dots, C_t$ where $|W_i : = E(C_i) \setminus (\cup_{j=1}^i C_j \cup B)| \le 2$, and $T' \subseteq \cup_{i= 1}^t W_i$.
Moreover, because $F \cap \langle B \rangle_2 = \emptyset$, it follows that $F \cap E(C_i) = \emptyset$ for each $i$.
Working from $C_t$ to $C_1$, we choose a flow $f_i: E(C_i) \rightarrow \Z_p$ so that $\phi_1 := f^* + \sum_{i=1}^t f_i$ is a flow where $\phi_1(e) \not \in Y(e)$ for every $e \in T'$.
As before, every $f_i$ sets the value $\phi_1(e)$ for each $e \in W_i$. For every $f_i$ there are $p$ choices of flow value, five of which are forbidden for each $e \in W_i$ -- they would cause $\phi_1(e) \in Y(e)$. Since $|W_i| \le 2$, this means there at most ten forbidden values. Since $p \ge 11$, we can always choose an appropriate $f_i$ so that $\phi_1(e) \notin Y(e)$ for all $e \in W_i$.

Let $B_1:=\{e \in B : \phi_1(e) = \overline{f}(e)\}$. We now aim to construct an integer 3-flow $\psi: T \cup B_1 \to \Z$ such that $\psi(e)=\pm 1$ for all $e\in B_1$.
To this end, observe that adding each edge $e\in B_1$ to $T$ creates either a barbell or a positive cycle. In the former case, add the two cycles of the barbell to a set $\mathcal{C}$; in the latter case add the positive cycle to $\mathcal{C}$. In both cases, an even number of negative edges are added to $\mathcal{C}$. Now let $H$ be the graph induced by the symmetric difference of all the cycles in $\mathcal{C}$ and note that the graph $H$ has an even number of negative edges. Note also that because every $e \in B_1$ is in exactly one cycle of $\mathcal{C}$, it follows that $B_1 \subseteq E(H)$. Since $H$ is the symmetric difference of cycles, it has a nz $\mathbb{Z}_2$-flow. We can in turn view this as a $\mathbb{Z}_2$-flow of $T \cup B_1$ with support $E(H)$. By Theorem \ref{lemma2to3}, this means there is a 3-flow $\psi: T \cup B_1 \to \Z$ such that $\psi(e)=\pm 1$ for all $e\in B_1$.

Now consider $\phi:=\phi_1 + 3\psi$ where the sum is performed in $\Z_p$.
It is possible that $\phi(e') =\overline{f}(e')$. But since $\phi_1(e')\neq \overline{f}(e')$, this would imply that $\psi(e') \neq 0$. If this is the case, then define $\phi=\phi_1 - 3\psi$ instead. Note that $3\psi(e')\neq -3\psi(e')$ since this would mean that 12 is congruent to 0 modulo $p$, which is not the case for $p \ge 11$ and prime. Thus we get $\phi(e') \neq\overline{f}(e')$; it remains to check that $\phi(e) \neq \overline{f}(e)$ for every $e \in E(G) \setminus \{e'\}$.

Let $e \in E(G) \setminus \{e'\}$. It is the case that $3\psi(e)\in\{0, \pm 3, \pm 6\}$. If $e\in T \setminus \{e'\}$, then by construction $\phi_1(e) \not \in \{\overline{f}(e), \overline{f}(e) \pm 3, \overline{f}(e) \pm 6\}$, and so $\phi(e) \neq \overline{f}(e)$.
Now suppose that $e\in B$. If $\phi_1(e)\neq \overline{f}(e)$, then $\psi(e)=0$ (since $\psi$ was not specified on these edges), so $\phi(e)=\phi_1(e)\neq \overline{f}(e)$. On the other hand, if $\phi_1(e)=\overline{f}(e)$, that is $e\in B_1$, then $3\psi(e)=\pm 3$, so $\phi(e)= \overline{f}(e) \pm 3 \neq \overline{f}(e)$.
\end{proof}

\section{Acknowledgements}

We thank Matt DeVos for helpful discussions.

\bibliographystyle{amsplain}
\bibliography{bib}

\end{document}